\documentclass{amsart}
\usepackage{soul}
\usepackage{amssymb}
\usepackage{stmaryrd} 
\usepackage{amsmath} 
\usepackage{amscd}
\usepackage{cancel}
\usepackage{amsbsy}
\usepackage{commath}
\usepackage{comment, enumerate}
\usepackage{geometry}
\usepackage[matrix,arrow]{xy}
\usepackage{hyperref}
\hypersetup{
 colorlinks=true,
 linkcolor=blue,
 filecolor=magenta,
 citecolor=brown,
 urlcolor=cyan,
 pdftitle={Non-Vanishing mod p of Hecke L-values},
 }
\usepackage{mathrsfs}
\usepackage{color}
\usepackage{mathtools,caption}
\usepackage{tikz-cd}
\usepackage{longtable}
\usepackage[utf8]{inputenc}
\usepackage[OT2,T1]{fontenc}
\usepackage[normalem]{ulem}

\DeclareSymbolFont{cyrletters}{OT2}{wncyr}{m}{n}
\DeclareMathSymbol{\Sha}{\mathalpha}{cyrletters}{"58}

\DeclareMathOperator{\Hom}{Hom}
\DeclareMathOperator{\Gal}{Gal}
\DeclareMathOperator{\alg}{(alg)}
\DeclareMathOperator{\cyc}{cyc}
\DeclareMathOperator{\ord}{ord}

\DeclareMathOperator{\Tr}{Tr}

\DeclareMathOperator{\charac}{char}

\DeclareMathOperator{\End}{End}
\DeclareMathOperator{\Aut}{Aut}

\newcommand{\QQ}{\mathbb Q}
\newcommand{\cD}{\mathcal D}
\newcommand{\FF}{\mathbb F}
\newcommand{\GG}{\mathbb G}
\newcommand{\ZZ}{\mathbb Z}
\newcommand{\CC}{\mathbb C}

\newcommand{\cF}{\mathcal F}

\newcommand{\cO}{\mathcal O}
\newcommand{\fa}{\mathfrak a}
\newcommand{\ff}{\mathfrak f}
\newcommand{\Djk}{\mathscr{D}_{j,k}}
\newcommand{\Dnotk}{\mathscr{D}_{0,k}}
\newcommand{\fb}{\mathfrak b}
\newcommand{\fm}{\mathfrak m}
\newcommand{\fc}{\mathfrak c}
\newcommand{\fp}{\mathfrak p}
\newcommand{\fq}{\mathfrak q}
\newcommand{\fg}{\mathfrak g}
\newcommand{\fh}{\mathfrak h}
\newcommand{\fl}{\mathfrak{l}}

\newcommand{\cL}{\mathcal{L}}

\newcommand{\sR}{\mathscr{R}}

\newcommand{\Zq}{\mathbb{Z}_q}
\newcommand{\fO}{\mathfrak{O}}
\newcommand{\rational}{\vartheta_{\fa,V}^\Psi}

\numberwithin{equation}{section}
\newtheorem*{Theorem*}{Theorem}
\newtheorem{Th}{Theorem}[section]
\newtheorem{Lemma}[Th]{Lemma}
\newtheorem*{Ques*}{Question}
\newtheorem{Prop}[Th]{Proposition}

\newtheorem*{conj*}{Conjecture}
\newtheorem{lthm}{Theorem}

\definecolor{Green}{rgb}{0.0, 0.5, 0.0}

\newtheorem{rem}[Th]{Remark}

\newtheorem{Defi}[Th]{Definition}

\begin{document}
\title[Non-vanishing mod $p$ of Hecke $L$-values]{Non-vanishing modulo $p$ of Hecke $L$-values over imaginary quadratic fields}

\author[D.~Kundu]{Debanjana Kundu}
\address[Kundu]{Fields Institute \\ University of Toronto \\ Toronto ON, M5T 3J1, Canada}
\email{dkundu@math.toronto.edu}

\author[A.~Lei]{Antonio Lei}
\address[Lei]{Department of Mathematics and Statistics\\University of Ottawa\\
150 Louis-Pasteur Pvt\\
Ottawa, ON\\
Canada K1N 6N5}
\email{antonio.lei@uottawa.ca}

\date{\today}


\begin{abstract}
Let $p$ and $q$ be two distinct odd primes.
Let $K$ be an imaginary quadratic field over which $p$ and $q$ are both split.
Let $\Psi$ be a Hecke character over $K$ of infinity type $(k,j)$ with $0\le-j< k$.
Under certain technical hypotheses, we show that for a Zariski dense set of finite-order characters $\kappa$ over $K$ which factor through the $\Zq^2$-extension of $K$, the $p$-adic valuation of the algebraic part of the $L$-value $L(\overline{\kappa\Psi},k+j)$ is a constant independent of $\kappa$.
In addition, when $j=0$ and certain technical hypothesis holds, this constant is zero.
\end{abstract}

\maketitle

\section{Introduction}
Let $p$ and $q$ be two distinct odd primes.
It is a classical problem to study the divisibility of the algebraic part of (Hecke) $L$-values by a given prime $p$ as one varies the (Hecke) characters of $q$-power conductor.
For Dirichlet $L$-values, such questions were studied by L.~Washington in \cite{Was75, Was78}.
He showed that for almost all Dirichlet characters of $q$-power conductor, the algebraic parts of their $L$-values are coprime to $p$.
As an application, he proved that the $p$-part of the class number stabilizes in cyclotomic $\ZZ_q$-extensions of abelian number fields.
Washington's results have been extended to the case of (finite) product of cyclotomic $\mathbb{Z}_{q_i}$-extensions of abelian number fields (for distinct primes $q_i$ with $q_i\ne p$) by E.~Friedman in \cite{friedman}.

In \cite{sinnott}, W.~Sinnott introduced the idea of relating non-vanishing of such $L$-values modulo $p$ to Zariski density (modulo $p$) of special points of the algebraic variety underlying the $L$-values.
Using this machinery, J.~Lamplugh generalized Washington's theorem to \emph{split prime} $\ZZ_q$-extensions of imaginary quadratic fields in \cite{Lam15}.
Let $K$ be an imaginary quadratic field such that $q \cO_K= \fq{\fq^*}$ with $\fq \neq \fq^*$, then the \emph{split prime} $\Zq$-extension of $K$ is one where only one of $\fq$ or ${\fq^*}$ is ramified.


In \cite{hida1,hida2}, H.~Hida studied analogous questions for anticyclotomic characters.
He proved that when $p$ is split in $K$ and the tame conductor of characters is a product of split primes (which excludes the self-dual characters), the algebraic parts of the $L$-values of "almost all" anticyclotomic characters of $q$-power conductor over a CM field are non-zero mod $p$.
Here, "almost all" means "Zariski dense" after identifying the characters with a product of the multiplicative group (see Remark~\ref{rk:Zar}).
This has been generalized by M.-L.~Hsieh to include self-dual characters assuming that $p$ is split in $K$ in \cite{hsieh}  and that the inert part of the conductor is square-free.
The hypothesis on the inert part of the conductor was removed in \cite[Remark~6.4]{hsieh2014mu}.
In the case where the CM field is an imaginary quadratic field, T.~Finis has proved similar results for self-dual characters allowing $p$ to be either inert or ramified in $K$, and has determined precisely the $p$-adic valuations of the algebraic parts of anticyclotomic Hecke characters of $q$-power conductor; see \cite{finis}.
More recently, A.~Burungale showed in \cite{burungale} that these results may be extended to Hida families of anticyclotomic characters under the same hypotheses as those in the works of Hsieh.

We study a generalization of the aforementioned results on anticyclotomic characters to Hecke characters (not necessarily anticyclotomic) of $q$-power conductor over an imaginary quadratic field.
It can be regarded as a 2-variable version of \cite[Theorem~6.9]{Lam15}.

\begin{lthm}\label{thmA}
Let $K$ be an imaginary quadratic field over which $p$ and $q$ are both split.
Suppose that both the prime ideals above $q$ are principal in $K$.
Let $\Psi$ be a Hecke character over $K$ of infinity type $(k,j)$ and conductor $\ff$, where $0\le -j<k$ and $\ff$ is coprime to $pq$.
Assume that $q\nmid [\sR(\ff):K]$, where $\sR(\ff)$ denotes the ray class field of $K$ of conductor $\ff$.
Let $\cF_\infty$ be the $\Zq^2$-extension of $K$.
There exists a constant $C_\Psi$ such that for a Zariski dense set of finite-order characters $\kappa$ of $\Gal(K_\infty/K)$,
\[
\ord_{p}\left( L^{\alg} \left(\overline{\kappa\Psi}\right)\right)=C_\Psi.
\]
\end{lthm}

Under additional hypotheses, we prove:
\begin{lthm}\label{thmB}
With notation as in the statement of Theorem~\ref{thmA}, if $j=0$ and the character of $\Gal(\sR(\ff)/K)$ induced by $\Psi$ satisfies a technical hypothesis \eqref{eq:condition-rho}, then $C_\Psi=0$.
\end{lthm}

\begin{rem}
If $p\nmid [\sR(\ff):K]$, then it is easy to show (see Remark~\ref{remark: sufficient conditions for technical hypothesis}) that one may multiply $\Psi$ by a character $\rho$ of $\Gal(\sR(\ff)/K)$ such that the technical hypothesis is satisfied.
\end{rem}

\subsection*{Outline of proofs} The proofs of Theorems~\ref{thmA} and \ref{thmB} follow closely the line of argument of \cite[Theorem~6.9]{Lam15}.
It consists of the following ingredients:
\begin{itemize}
 \item[(1)]Establish a theory of Gamma transform of "elliptic function measures" on $\Zq^2$, which are measures that arise from a rational function on an elliptic curve.
 \item[(2)]Show that the $\pi$-adic valuations of the aforementioned Gamma transforms have the same $p$-adic valuation for almost all finite characters on $\Zq^2$.
 \item[(3)]Show that by defining an elliptic function measure (see Definition~\ref{def:elliptic-fun-meas}) arising from a rational function on the CM elliptic curve $E$ attached to $\Psi$, the Gamma transforms of this measure is related to the special values of $L$-series that we are interested in.
This proves Theorem~\ref{thmA}.
\item[(4)]To prove Theorem~\ref{thmB}, we show that the $\pi$-adic valuation discussed in (2) is zero under our additional hypotheses.
\end{itemize}
Step (1) is carried out in Section~\ref{section: distribution and measures}.
We follow the strategy of Lamplugh in \cite[Section~3]{Lam15}, where the theory for elliptic function measures on $\Zq$ was developed.
To execute (2), we use a lemma of Hida on the Zariski density of characters on $\Zq^d$ from \cite{hida1} to prove a result on the algebraic independence of functions on elliptic curves with positive characteristic.
In particular, we prove Theorem~\ref{thm:Hida-Lam}, which is a two-variable version of \cite[Theorem~4.9]{Lam15}.
Next, we prove Theorem~\ref{thm:valuation-gamma-transform}, which completes step (2) outlined above.
The corresponding 1-variable version of this theorem was proved in \cite[Section~5]{Lam15}.
The construction of the elliptic function measure of step (3) is discussed in Section~\ref{S:elliptic-function}; this is a generalization of the rational function on the CM elliptic curve $E$ utilized in \cite[Section~6.3]{Lam15} and crucially uses the work of E.~de Shalit \cite{dS87}.
The link between the Gamma transforms of this elliptic function measure and the $L$-values of interest is given by Lemma~\ref{thm: Lam15 6.6}.
In step (2), we see that the $\pi$-adic valuation mentioned above is in fact given by the valuation of the rational function (see Definition~\ref{defn:ord}).
Using ideas of the proof of \cite[Lemma~6.7]{Lam15} in the one-variable case, we show in Lemma~\ref{lemma: Lam15 Lemma 6.7} that this valuation is zero; this allows us to conclude step (4).

\begin{rem}
While Theorems~\ref{thmA} and \ref{thmB} are deduced using Lamplugh's techniques developed in \cite{Lam15}, our results are strictly stronger than the one-variable analogue \cite[Theorem~6.9]{Lam15}.
Indeed, after identifying the characters of $\Gal(F_\infty/F)$ with a subset of $\GG_{m/\overline{\QQ}_q}^2$, the Zariski closure of the set of characters given by \emph{loc. cit.} is one copy of $\GG_{m/\overline{\QQ}_q}$.
In particular, it is not Zariski dense in $\GG_{m/\overline{\QQ}_q}^2$.

Furthermore, we consider Hecke characters of much more general infinity type than the ones considered in \cite{Lam15}.
In addition, the class number of $K$ is assumed to be $1$ in \cite{Lam15}, whereas Theorem~\ref{thmA} assumes that $q$ does not divide {$ [\sR(\ff):K]$ instead.}
\end{rem}

\begin{rem}
Using an argument similar to the one presented in \cite[Section~7]{Lam15}, we expect that {Theorem~\ref{thmB}} combined with the Iwasawa main conjecture (proved by K.~Rubin) should show that the $p$-part of the class groups over a $\Zq^2$-tower is "generically zero".
However, it does not seem to be enough to give a generalization of \cite[Theorem~7.10]{Lam15} in our setting, unless we replace "almost all" by "all but finitely many".
\end{rem}

\noindent We conclude by discussing some follow-up questions.
\begin{itemize}
\item In \cite{KL2}, we study the growth of the $p$-part of the class groups in the anticyclotomic $\Zq$-extension making use of the aforementioned result of Hida.
\item Similar to how we build on Lamplugh's results to obtain {our results}, it may be possible to prove a similar result for Hecke characters of $q$-power conductor over general CM fields, relying on results of Hida and Hsieh on anticyclotomic characters.
\item It may also be interesting to generalize {our results} to the setting of Hida families, utilizing ideas of Burungale developed in \cite{burungale}.
\end{itemize}

\section*{Acknowledgements}
DK thanks Jack Lamplugh for providing a copy of his thesis.
AL thanks Ashay Burungale for helpful discussions and for his comments on an earlier draft.
DK is supported by the PIMS postdoctoral fellowship.
AL is supported by the NSERC Discovery Grants Program RGPIN-2020-04259 and RGPAS-2020-00096.
This work was initiated during the thematic semester "Number Theory -- Cohomology in Arithmetic" at Centre de Recherches Math{\'e}matiques (CRM) in Fall 2020.
The authors thank the CRM for the hospitality and generous supports.
Finally, we thank the referee for their comments on earlier versions of the article and their valuable suggestions, which led to the removal of several technical hypotheses from our main results.

\section{Basic Notions}
\label{section: preliminaries}
Let $K$ be a fixed imaginary quadratic field of discriminant $d_K$ and $H$ denote its Hilbert class field.
Throughout, we assume that $q$ is coprime to the class number of $K$ {and we fix a Hecke character $\Psi$ given as in the statement of Theorem~\ref{thmA}.
The character $\overline{\Psi}\circ N^{-j}$ (where $N$ is the norm map on $K$) is of infinity-type $(0,k-j)$.
There exists a character $\chi_0$ of $\Gal(\sR(\ff)/K)$ and an elliptic curve $E$ defined over $\sR(\ff)$ with complex multiplication by $\cO_K$, i.e., $\cO_K \simeq \End(E)$, such that 
\[
\overline{\Psi} N^{-j}=\overline{\varphi^{k-j}}\chi_0,
\]
where $\varphi$ is a Hecke character of infinity type $(1,0)$ satisfying
\[
\psi=\varphi\circ N_{\sR(\ff)/K}
\]
with $\psi$ being the Hecke character over $\sR(\ff)$ attached to $E$.
Furthermore, $\sR(\ff)(E_{\mathrm{tor}})$ is an abelian extension of $K$.
(See \cite[Chapter II, proofs of Theorems~4.12 and 4.14]{dS87} where the existence of $E$ and $\chi_0$ is discussed.)
} 
Let $q\ge 5$ be a prime number that splits in $K$, i.e., 
\[q\cO_K = \fq {\fq^*} \textrm{ with } \fq\neq {\fq^*}.\]

For any integral ideal $\fa $ in $\cO_K$, we write $E_{\fa}$ to denote
\[\ker\left( \fa: E \rightarrow E\right).\]
We write $\mu_K$ to denote the set of roots of unity in $K$ and $w_K$ to denote the size of this set.

We fix a different prime $p$ {such that $p \cO_K= \fp {\fp^*}$ in $K$ with $\fp \neq \fp^*$ and $\gcd(p, 6\ff q)=1$.
Note in particular that $E$ has good reduction at all primes above $pq$.}
\section{Distributions and measures on \texorpdfstring{$\Zq^2$}{}}
\label{section: distribution and measures}
The goal of this section is to generalize the notion of Gamma transform from \cite{sinnott} and elliptic function measures studied in \cite[Section~3.2]{Lam15} to the two-variable setting.

Let $E$ be {the elliptic curve given in \S\ref{section: preliminaries}} and $k/\QQ_p$ be a finite unramified extension containing $\QQ_p(E_{\mathfrak{f}q})$.
Set $J=k(\mu_{q^\infty})$.
This is the unramified $\Zq$-extension of $k$ (since $\mu_{q}\subset k$ by assumption).
Let $\fO$ denote the ring of integers of $J$.
Fix a uniformizer $\pi$ of $k$ and let $\ord_{\pi}$ denote the normalized valuation map
\[
\ord_{\pi}: J \rightarrow \ZZ \cup \{ \infty\}.
\]

\begin{Defi}
\label{def:dist}
Let $\alpha$ be a \emph{$J$-valued distribution} on $\Zq^2$, i.e., $\alpha$ is a finitely additive function on the set of compact open subsets of $\ZZ_q^2$ with values in $J$.
\begin{itemize}
 \item[(i)] Given any $c=(c_1,c_2)\in(\Zq^\times)^2$, define $\alpha\circ c$ to be the distribution given by $\alpha\circ c(X)=\alpha(c X)$ for all open compact subsets $X$ of $\Zq^2$.
 \item[(ii)] The \emph{Fourier transform} of $\alpha$ is defined to be
 \begin{align*}
 \hat{\alpha}:\mu_{q^\infty}^2&\rightarrow J\\
 (\zeta_1,\zeta_2)&\mapsto \int_{(x,y)\in\Zq^2}\zeta_1^x\zeta_2^yd\alpha(x,y).
\end{align*}
 \item[(iii)] Given a finite character $\chi $ on $(\Zq^\times)^2$ with values in $J$, we define \emph{Leopoldt's $\Gamma$-transform} as
 \[
 \Gamma_\alpha(\chi)=\int_{\Zq^2}\chi d\alpha,
 \]
 where we extend $\chi$ to $\Zq^2$ by sending all elements not inside $(\Zq^\times)^2$ to zero.
 \item[(iv)]We call $\alpha$ a \emph{measure} on $\ZZ_q^2$ if the image of $\alpha$ has bounded values with respect to $\ord_\pi$.
\end{itemize}
\end{Defi}

\begin{Lemma}
\label{lem:Gamma-Fourier}
Suppose that $\chi$ is a finite-order character on $(\ZZ_q^\times)^2$ factoring through $(\ZZ/q^m)^\times \times (\ZZ/q^{n})^\times$, then
\[
\Gamma_\alpha(\chi)=\tau(\chi)\sum_{\underline x\in \ZZ/q^m\times\ZZ/q^n}\chi^{-1}(\underline x)\hat\alpha(\underline\zeta^{\underline x}),
\]
where $\underline\zeta=(\zeta_m,\zeta_n)$ with $\zeta_m$ and $\zeta_n$ being primitive $p^m$-th and $p^n$-th roots of unity respectively, and $\tau(\chi)$ is the \emph{Gauss sum} of $\chi$ defined by
\[
\tau(\chi)=\frac{1}{q^{m+n}}\sum_{ (x_1,x_2)\in \ZZ/q^m\times \ZZ/q^n}\chi( x_1,x_2)\zeta_m^{-x_1}\zeta_n^{-x_2}.
\]
\end{Lemma}

\begin{proof}
See \cite[proof of Proposition~2.2, equation (2.6)]{sinnott} (or \cite[proof of Lemma~2.2.3]{Lam-thesis}).
\end{proof}

\begin{Lemma}
\label{lem:dist-fourier}
A distribution $\alpha$ on $\Zq^2$ is uniquely determined by its Fourier transform $\hat\alpha$.
\end{Lemma}

\begin{proof}
The characteristic function on the open subset $U_{a,b} := (a+q^m\Zq)\times (b+q^n\Zq)$ of $\Zq^2$ satisfies
\[
\mathbf{1}|_{U_{a,b}}=\frac{1}{q^{m+n}}\sum_{(\zeta_1,\zeta_2)\in\mu_{q^m }\times\mu_{q^n}}\zeta_1^{-a}\zeta_2^{-b}\chi_{(\zeta_1,\zeta_2)},
\]
where $\chi_{(\zeta_1,\zeta_2)}:\Zq^2\mapsto J$ is the character sending $(x,y)$ to $\zeta_1^{x}\zeta_2^{y}$.
In particular, we see that $\alpha\left(U_{a,b}\right)$ is a linear combination of $\hat\alpha(\zeta_1,\zeta_2)$.
Since the subsets $U_{a,b}$ form a basis of open compact sets of $\Zq^2$, $\alpha$ is uniquely determined by $\hat\alpha$.
\end{proof}

For the rest of the article, we fix an isomorphism of groups $\delta:(\mu_{q^\infty})^2\stackrel{\sim}{\longrightarrow} E_{q^\infty}$.

\begin{Defi}\label{defn:ellipticmeasure}
A $J$-valued distribution $\alpha$ on $\Zq^2$ is an \emph{elliptic function measure} for our fixed elliptic curve $E$ (with respect to $\delta$) if there exists a rational function $R\in J(E)$ such that for almost all $\underline\zeta\in(\mu_{q^\infty})^2$, we have
\[
\hat\alpha(\underline\zeta)=R(\delta(\underline\zeta)).
\]
\end{Defi}

\begin{Lemma}
\label{lem:integral-elliptic-function}
Let $f\in \fO[x,y]$ such that the image of $f$ in $J(E)$ is non-zero.
Then, there exists a unique integer $n\ge0$ such that 
\[
\ord_\pi(f(Q))\ge n\ \forall Q\in E_{q^\infty}\setminus\{0\}
\]
with equality holding for almost all $Q\in E_{q^\infty}$.
\end{Lemma}

\begin{proof}
The proof of \cite[Lemma~3.2]{Lam15} goes through in verbatim on replacing $E_{\fq^\infty}$ by $E_{q^\infty}$.
\end{proof}

This lemma allows us to define a valuation on $J(E)$.
\begin{Defi}\label{defn:ord}
Given an $R\in J(E)$.
If $R\ne 0$, we define $\ord_\pi(R)$ to be the integer $n $ such that $\ord_\pi(R(Q))=n$ for almost all $Q\in E_{q^\infty}$.
If $R=0$, we set $\ord_\pi(R)=\infty$.
\end{Defi}

By Lemma~\ref{lem:integral-elliptic-function}, if $\alpha$ is an elliptic function measure, then it is in fact a measure (not just a distribution) since the values of $\alpha$ are linear combinations of $\hat\alpha=R\circ \delta$ as we have seen in the proof of Lemma~\ref{lem:dist-fourier} and $\frac{1}{q^{m+n}}\in \fO^\times$ (as $p\ne q$).

Note that for any given rational function $R\in J(E)$, we can define a $J$-valued measure attached to $R$ as given by the following lemma:

\begin{Lemma}
\label{lem:EFM}
Let $R\in J(E)$ be a rational function.
There exists a unique measure $\alpha$ on $\Zq^2$ such that the Fourier transform $\hat\alpha$ coincides with $R\circ \delta$.
In other words, $\alpha$ is an elliptic function measure associated to $R$ in the sense of Definition~\ref{defn:ellipticmeasure}.
\end{Lemma}

\begin{proof}
By the proof of Lemma~\ref{lem:dist-fourier}, we may define a measure $\alpha$ satisfying
\[
\alpha\left((a+q^m\Zq)\times (b+q^n\Zq)\right)=\frac{1}{q^{m+n}}\sum_{(\zeta_1,\zeta_2)\in\mu_{q^m }\times\mu_{q^n}}\zeta_1^{-a}\zeta_2^{-b}R\circ \delta (\zeta_1,\zeta_2).
\]
It follows from direct calculations that $\alpha$ is additive and that $\hat\alpha =R\circ \delta$.
\end{proof}

We now show how Gamma transforms behave under Galois actions.
This will be utilized in subsequent sections.
Let us define the following homomorphisms of groups
\begin{align*}
 \chi_\mu&:\Gal(J/k)\hookrightarrow \Aut(\mu_{q^\infty})^2\simeq (\Zq^\times)^2,\\
 \chi_E&:\Gal(J/k)\hookrightarrow \Aut(E_{\fq^\infty})\times\Aut(E_{\overline\fq^\infty})\simeq (\Zq^\times)^2.
\end{align*}
Note that $\chi_\mu=\chi_{\cyc}\times\chi_{\cyc}$, where $\chi_{\cyc}$ is the cyclotomic character.

\begin{Defi}\label{def:elliptic-fun-meas}
An \textbf{elliptic function measure} $\alpha$ for $E$ is said to be \emph{defined over $k$}, if $\hat\alpha=R\circ\delta$ for a rational function $R\in k(E)$.
\end{Defi}

\begin{Lemma}
\label{lem:Gamma-sigma}
Suppose that $\alpha$ is an elliptic function measure defined over $k$.
Then, for almost all finite-order characters $\kappa$ of $(\Zq^\times)^2$ and for all $\sigma\in\Gal(J/k)$, we have
\[
\Gamma_\alpha(\kappa)^ \sigma=\frac{\kappa^\sigma(\chi_E(\sigma))}{\kappa^\sigma(\chi_\mu(\sigma))}\Gamma_\alpha(\kappa^\sigma).
\]
\end{Lemma}

\begin{proof}
It follows from Lemma~\ref{lem:Gamma-Fourier} that
\[
\Gamma_\alpha(\kappa)^\sigma=\tau(\kappa)^\sigma\sum_{\underline x\in \ZZ/q^m\times\ZZ/q^n}\kappa^{-1}(\underline x)^\sigma\hat\alpha(\underline\zeta^{\underline x})^\sigma.
\]
We have
\begin{align*}
\tau(\kappa)^\sigma&=\frac{1}{q^{m+n}}\sum_{(x_1,x_2)\in \ZZ/q^m\times \ZZ/q^n}\kappa(x_1,x_2)^\sigma(\zeta_m^{-x_1}\zeta_n^{-x_2})^\sigma\\
&=\frac{1}{q^{m+n}}\sum_{(x_1,x_2)\in \ZZ/q^m\times \ZZ/q^n}\kappa(x_1,x_2)^\sigma\zeta_m^{-\chi_{\cyc}(\sigma)x_1}\zeta_n^{-\chi_{\cyc}(\sigma)x_2}\\
&=\frac{\kappa^\sigma(\chi_{\cyc}(\sigma),\chi_{\cyc}(\sigma))^{-1}}{q^{m+n}}\sum_{(x_1,x_2)\in \ZZ/q^m\times \ZZ/q^n}\kappa(x_1,x_2)^\sigma\zeta_m^{-x_1}\zeta_n^{-x_2}\\
&=\kappa^\sigma(\chi_\mu(\sigma))^{-1}\tau(\kappa^\sigma).
\end{align*}
Since $\alpha$ is an elliptic function measure, we have
\[
\hat\alpha(\underline\zeta^{\underline x})^\sigma=R(\delta(\underline\zeta^{\underline x})^\sigma)=R(\delta(\underline\zeta^{\chi_E(\sigma)\underline x}))=\hat\alpha(\underline\zeta^{\chi_E(\sigma)\underline x})
\]
for some $R\in k(E)$.
Therefore, combining these equations gives
\begin{align*}
\Gamma_\alpha(\kappa)^\sigma&= \kappa^\sigma(\chi_\mu(\sigma))^{-1}\tau(\kappa^\sigma) \sum_{\underline x\in \ZZ/q^m\times\ZZ/q^n}\kappa^{-1}(\underline x)^\sigma\hat\alpha(\underline\zeta^{\chi_E(\sigma)\underline x})\\
&=\frac{\kappa^\sigma(\chi_E(\sigma))}{\kappa^\sigma(\chi_\mu(\sigma))}\tau(\kappa^\sigma) \sum_{\underline x\in \ZZ/q^m\times\ZZ/q^n}\kappa^{-1}(\underline x)^\sigma\hat\alpha(\underline\zeta^{\underline x})\\
&=\frac{\kappa^\sigma(\chi_E(\sigma))}{\kappa^\sigma(\chi_\mu(\sigma))}\Gamma_\alpha(\kappa^\sigma)
\end{align*}
where the last equality follows from Lemma~\ref{lem:Gamma-Fourier} applied to $\kappa^\sigma$.
\end{proof}

\section{Algebraic Independence Results}
The main result of this section is Theorem~\ref{thm:Hida-Lam}, where we prove an algebraic independence result of functions on $E_{q^\infty}$ taking values in a \emph{finite field} whose characteristic is distinct from $q$.
The first step is Theorem~\ref{thm: Lamplugh 4.5}, which is an analogue of \cite[Proposition~3.1]{sinnott} (and also \cite[Theorem~4.5]{Lam15}).
This step involves proving an algebraic independence result of functions on $E_{q^\infty}$ taking values in a \emph{general field}, $\cF$.
Let $E$ be an elliptic curve as fixed in Section~\ref{section: preliminaries}.
We suppose that $E$ can be considered as a curve over the field $\cF$ (for example, the residue field of $H$ modulo a prime ideal).
Suppose that $q>3$ is a rational prime that splits in $\cO_K$ and $\charac(\cF)\neq q$.
This result essentially says that endomorphisms in $\End(E_{\fq^\infty})\times \End(E_{{\fq^*}^\infty})$ which are independent over $\End_{\cF}(E)$, are algebraically independent.

The following lemma is required for the proof of Theorem~\ref{thm: Lamplugh 4.5}.

\begin{Lemma}
\label{prop: Lam15 4.4}
Let $\Phi_1, \ldots, \Phi_s$ be non-trivial morphisms from $E^n$ to $E$ of the form
\[
\Phi_j: (P_i)_{i=1}^n \mapsto \sum_{i=1}^n \alpha_{ij}(P_i)
\]
where $\alpha_{ij}\in \End_{\cF}(E)$ for all $1\leq i \leq n$ and $1\leq j \leq s$.
Suppose that the only relation of the kind $\alpha\Phi_k = \beta \Phi_{\ell}$ for $\alpha, \beta\in \End_{\cF}(E)$ and $k\neq \ell$, is when $\alpha = \beta =0$.
If $r_1, \ldots, r_s\in \cF(E)$ with $\sum_{j=1}^s r_j \circ \Phi_j = 0$, then each $r_j$ is a constant function.
\end{Lemma}

\begin{proof}
See {\cite[Proposition~4.4]{Lam15}}.
\end{proof}

\begin{Th}
\label{thm: Lamplugh 4.5}
Let $\cF$ be any field as above, and $E$ an elliptic curve defined over $\cF$ such that $\End_{\cF}(E)\simeq \cO_K$.
Suppose that $\underline{\eta}_1, \ldots, \underline{\eta}_s \in \End(E_{\fq^\infty})\times \End(E_{{\fq^*}^\infty})$ are such that $\alpha\underline{\eta}_k = \beta\underline{\eta}_{\ell}$ for $k\neq \ell$ and some $\alpha, \beta\in \End_{\cF}(E)$ only when $\alpha=\beta=0$.
Consider the function 
\[
R = \sum_{j=1}^s r_j \circ \underline{\eta}_j: E_{\fq^\infty}\times E_{{\fq^*}^\infty}\rightarrow \overline{\cF} \]
where $r_j \in \cF(E)$ and $\overline{\cF}$ denotes an algebraic closure of $\cF$.
If $R(Q)=0$ for all $Q\in E_{\fq^\infty}\times E_{{\fq^*}^\infty}$, then all $r_j$'s are constant functions.
\end{Th}

\begin{proof}
We recall that $\End(E_{\fq^\infty})\simeq \cO_{\fq}$, $\End(E_{{\fq^*}^\infty})\simeq \cO_{{\fq^*}}$ and $\End_{\cF}(E)\simeq \cO_K$.
Consider a free $\cO_K$ submodule $A$ of $\cO_{\fq}\times\cO_{{\fq^*}}$ of rank $n$ that contains $\underline{\eta}_j$ for $1\leq j\leq s$.
Let $\{\underline{\varepsilon}_{i}\}_{i=1}^n$ be an $\cO_K$-basis of $A$.
Then, there exist unique $\alpha_{ij}\in \cO_K$ such that
\[
\underline{\eta}_j = \sum_{i=1}^n \alpha_{ij}\underline{\varepsilon}_i.
\]
Define the map 
\[
\iota: E_{\fq^\infty}\times E_{{\fq^*}^\infty}= E_{q^\infty}\rightarrow E^{n}\textrm{ given by } Q \mapsto \left(\underline{\varepsilon}_i Q \right)_{i=1}^n.\]
For each $1\leq j \leq s$, denote the morphism
\[
\Phi_j : E^{n} \rightarrow E; \qquad \left( P_i\right)_{i=1}^n \mapsto \sum_{i=1}^n \alpha_{ij}P_i.
\]
We have assumed that
\[
\sum_{j=1}^s r_j \circ \Phi_j(\mathcal{Q}) =0 \qquad \textrm{for all }\mathcal{Q}\in \iota(E_{\fq^\infty}\times E_{{\fq^*}^\infty})\subseteq E^n.
\]
Hence, the above equality must hold for all $\mathcal{Q}$ in the Zariski closure of $\iota(E_{\fq^\infty}\times E_{{\fq^*}^\infty})$.
It follows from basic facts about Zariski closed subgroups of $E^{n}$ (see \cite[Lemmas~1 and 3]{Sch87}) that either the Zariski closure of $\iota(E_{q^\infty})$ is $E^{n}$ or there exist $\alpha_i \in \cO_K$ (not all zero) such that
\[
\sum_{i=1}^n \alpha_i \underline{\varepsilon}_i (Q) = 0 \textrm{ for all }Q\in E_{q^\infty}.
\]
If the latter holds, it means that $\sum_{i=1}^n \alpha_i \underline{\varepsilon}_i =0$.
However, this contradicts the fact that $\underline{\varepsilon}_1, \ldots ,\underline{\varepsilon}_n$ is a basis for $A$.
Thus, the Zariski closure of $\iota(E_{q^\infty})$ is $E^{n}$.
Lemma~\ref{prop: Lam15 4.4} implies that each $r_i$ is a constant function.
\end{proof}

To prove the main result in this section, we need a strengthened version of Theorem~\ref{thm: Lamplugh 4.5}.
This is achieved by combining the following Diophantine approximation result (Lemma~\ref{lem:approx}) with a special case of a lemma due to Hida (Lemma~\ref{lem:Hida}), which we record below.

\begin{Lemma}
\label{lem:approx}
Given $\underline{\beta}_1, \ldots, \underline{\beta}_d\in\cO_{\fq}\times \cO_{{\fq^*}}$ for any integer $d\geq 1$, and a positive constant $c\leq 1$, there exists an integer $N$ such that for all $n\geq N$, there exist algebraic integers $b_1, \ldots, b_d\in \cO_K$ and a unit $u \in \cO_{\fq}^{\times}\times \cO_{{\fq^*}}^\times$ satisfying
\begin{align*}
 v_{\fp} (u \underline{\beta}_i - b_i)&\geq n \textrm{ for }\fp\in \{\fq, {\fq^*}\} \textrm{ and}\\
 N_{K/\QQ}(b_i) &< c \cdot q^{2n}.
\end{align*}
\end{Lemma}
\begin{proof}
See {\cite[Lemma~2.3.9]{Lam-thesis}}.
\end{proof}

\begin{Lemma}
\label{lem:Hida}
Let $r$ be a positive integer.
Let $X=\bigcup_{i=1}^k X_i$ be a proper subset of $\GG_{m/\overline\QQ_q}^2$ such that
\begin{itemize}
 \item[(i)] $X$ is Zariski closed.
 \item[(ii)]For each $i$, there exists a closed subscheme $Y_i$ that is stable under $t\mapsto t^{p^{rn}}$ for all $n\in\ZZ$, such that $X_i=\underline{\zeta}Y_i$ for certain $\underline{\zeta}\in\mu_{q^\infty}^2$;
\end{itemize}
There exists $P$, which is a $p^r$-power, and an infinite sequence of integers $0<n_1<n_2<\cdots$ such that for all $j\ge1$,
\[
\Xi_j\cap X=\emptyset,
\]
where $\Xi_j$ is defined by
\[
\left\{\left(\frac{P^x}{q^{n_j}},\frac{P^y}{q^{n_j}}\right)\mod\Zq^2:x,y\in\ZZ\right\}\subset (\QQ_q/\Zq)^2
\]
after identifying $\mu_{q^\infty}^2$ with $(\QQ_q/\Zq)^2$ under an appropriate choice of basis.
\end{Lemma}

\begin{proof}
See {\cite[Lemma~3.4]{hida1}}.
\end{proof}
\begin{rem}\label{rk:Hida}On studying he proof of the above lemma, we see that $P\equiv 1 \mod q$.
If we write $P=1+q^vu$, where $q\nmid u$, then $\abs{\Xi_j}=q^{2(n_j-v)}$.\end{rem}

\begin{Th}\label{thm:Hida-Lam}
Let $\FF$ be a finite field.
Suppose that $\underline{\eta}_1, \ldots, \underline{\eta}_s \in \End(E_{\fq^\infty})\times \End(E_{{\fq^*}^\infty})$ are such that the only relation of the kind $\alpha\underline{\eta}_k = \beta\underline{\eta}_{\ell}$ for $k\neq \ell$ and $\alpha, \beta\in \End_{\FF}(E)$ is when $\alpha=\beta=0$.
Consider the function
\[
R = \sum_{i=1}^s r_i \circ \underline{\eta}_i: E_{\fq^\infty}\times E_{{\fq^*}^\infty}\rightarrow \overline{\FF}
\]
where $r_i \in \FF(E)$.
We identify $E_{q^\infty}$ with $\mu_{q^\infty}^2\subset \GG_{m/\overline\QQ_q}^2$.
Then either $\{Q\in E_{q^\infty}:R(Q)\ne 0\}$ is Zariski dense in $\GG_{m/\overline\QQ_q}^2$ or $R$ is identically zero.
In the latter case, all $r_i$'s are constant functions.
\end{Th}
\begin{proof}
Suppose that $\{Q\in E_{q^\infty}:R(Q)\ne 0\}$ is not Zariski dense and that $R$ is not identically zero.
We take $P$ to be a large enough $p$-power so that $R$ is defined over $\FF_P$ (the finite field of cardinality $P)$.
Let $X$ be the Zariski closure of $\{Q\in E_{q^\infty}:R(Q)\ne 0\}$ in $\GG_{m/\overline\QQ_q}^2$.
Then, $X$ is a proper subset of $\GG_{m/\overline\QQ_q}^2$ and $X^P\subseteq X$.

Let $\log_q:\widehat\GG_{m/\overline\QQ_q}^2\rightarrow \widehat\GG_{a/\overline\QQ_q}^2$ be the $q$-adic logarithm map.
We decompose $\log_q(X)$ into a finite union of closed subsets, each of which is stable under the multiplication by $P$.
This allows us to write $X$ as a finite union of closed subschemes $X_i$ of the form $\underline\zeta Y_i$, where $\underline\zeta \in \mu_{q^\infty}^2$ and $Y_i$ is stable under $t\mapsto t^P$.
Therefore, Lemma~\ref{lem:Hida} applies.
In particular, there exists a sequence of integers $0<n_1<n_2<\cdots$ and a collection of subsets $\Xi_j$ of $q^{n_j}$-torsion points in $E_{q^\infty}$ on which $R$ vanishes, with $\abs{\Xi_j}=q^{2(n_j-v)}$ for some fixed integer $v$.

Define 
\[
\delta := \max_{1\leq i \leq s}\deg (r_i).
\]
We apply Lemma~\ref{lem:approx} to $\underline{\eta}_1, \ldots, \underline{\eta}_s$ and $c =\displaystyle \frac{1}{q^{2v} \cdot s \cdot \delta }$.
There exists an integer $N$ such that for all $n_j\geq N$, there are algebraic integers $b_1, \ldots, b_s \in \cO_K$ and $u\in \cO_{\fq}^\times \times \cO_{{\fq^*}}^\times$ (depending on $n_j$) satisfying
\begin{align*}
 v_{\fp} (u \underline{\eta}_i - b_i)&\geq n_j \textrm{ for }\fp \in \{\fq, {\fq^*}\} \textrm{ and}\\
 \abs{ N_{K/\QQ}(b_i)} &< c \cdot q^{2n_j}.
\end{align*}
In particular, the rational function
\[
R_{n_j} := \sum_{i=1}^s r_i \circ b_i \in \FF(E)
\]
agrees with $R\circ u$ on $E_{q^{n_j}}$.
Thus, it vanishes on $\Xi_j$.
Moreover,
\[
\deg(R_{n_j})\le \sum_{i=1}^s\delta\cdot N_{K/\QQ}(b_i)<s\delta\cdot c\cdot q^{2n_j}=q^{2n_j-2v}=\abs{\Xi_j}.
\]
Therefore, $R_{n_j}=0$ and thus $R$ is zero on $E_{q^{n_j}}$.
But $n_j$ can be arbitrarily large.
This implies that $R$ is identically zero, which is a contradiction.
This concludes the first assertion of the theorem.
The last assertion follows immediately from Theorem~\ref{thm: Lamplugh 4.5}.
\end{proof}

\section{A theorem on two-variable Gamma transforms}
The purpose of this section is to prove a two-variable version of \cite[Theorem~5.1]{Lam15} (which in turn generalizes a result of Sinnott \cite[Theorem~3.1]{sinnott}).
Our proof utilizes crucially Theorem~\ref{thm:Hida-Lam} from the previous section.
Throughout, we use the same notation introduced in Sections~\ref{section: preliminaries} and \ref{section: distribution and measures}.

\begin{Th}\label{thm:valuation-gamma-transform}
Let $\alpha$ be an elliptic function measure for $E$ defined over $k$ on $\Zq^2$ that is supported on $(\Zq^\times)^2$, and satisfies $\alpha\circ\omega=\alpha$ for all $\omega\in\mu_K^2$.
Let $R$ denote the corresponding rational function (so that $\hat\alpha=R\circ \delta$ as in Definition~\ref{defn:ellipticmeasure}), and let $n=\ord_\pi(R)$ (as in Definition~\ref{defn:ord}).
Then for a Zariski dense set of finite-order characters $\kappa$ of $(1+q\Zq)^2$, we have
\[
\ord_\pi\left(\Gamma_\alpha(\kappa)\right)=n.
\]
\end{Th}
\begin{rem}\label{rk:Zar}
We view $\Hom(\Zq^2,\mu_{q^\infty})$ as a subset of $\GG_{m/\overline\QQ_q}^2$ by sending $\kappa$ to $(\kappa(1,0),\kappa(0,1))$.
A set of finite-order characters is called \emph{Zariski dense} if its image in $\GG_{m/\overline\QQ_q}^2$ is a dense subset under the Zariski topology.
\end{rem}

The following lemma is a key technical ingredient of the proof of Theorem~\ref{thm:valuation-gamma-transform}.

\begin{Lemma}
\label{lem:beta}
Let $\alpha$ be an elliptic function measure as in the statement of Theorem~\ref{thm:valuation-gamma-transform}.
Define
\[
\beta=\sum_\eta \left(\alpha\circ \eta\right)|_{(1+q\Zq)^2},
\]
where $\eta$ runs over a set of representatives for $(\mu_{q-1}/\mu_K)^2$ and $\alpha\circ\eta$ is defined as in Definition~\ref{def:dist}(i).
For each $y=(y_1,y_2){\in \mu_{q-1}^2}$, we write
\[
\beta_{y}=\beta|_{y_1(1+q^M\Zq)\times y_2(1+q^M\Zq)},
\]
where $M\ge1$ is the integer such that $\mu_{q^\infty}\cap k= \mu_{q^M}$.
Let $\kappa=(\kappa_1,\kappa_2)$ be a finite-order character of $(1+q\Zq)^2$.
Suppose that there exist integers $m,n\ge M$ satisfying
\[
\ker(\kappa_1)=1+q^{m+M}\Zq,\quad \ker(\kappa_2)=1+q^{n+M}\Zq.
\]
Let $\underline\zeta=(\zeta_1,\zeta_2)\in \mu_{q^\infty}$ such that 
\[
\zeta_1^{q^m}=\kappa_1(1+q^m),\quad \zeta_2^{q^n}=\kappa_2(1+q^n).
\]
Then $\Gamma_\beta(\kappa)\in\pi\fO$ if and only if $\hat\beta_{y}(\underline \zeta^{y^{-1}})\in \pi\fO$ for all $y\in \mu_{q-1}^2$.
\end{Lemma}

\begin{proof}
Suppose that $\Gamma_\beta(\kappa)\in\pi\fO$.
Let $\sigma\in \Gal(J/k)$ and $\underline\xi\in \mu_{q^\infty}^2$.
Recall from the proof of Lemma~\ref{lem:Gamma-sigma} that 
\[
\hat\alpha(\underline\xi)^\sigma=\hat\alpha(\underline\xi^{\chi_E(\sigma)}).
\]
Since Fourier transform is additive, we have equivalently
\[
\hat\beta(\underline\xi)^\sigma=\hat\beta(\underline\xi^{\chi_E(\sigma)}).
\]
Furthermore, Lemma~\ref{lem:Gamma-sigma} asserts that
\[
\Gamma_\beta(\kappa)^ \sigma=\frac{\kappa^\sigma(\chi_E(\sigma))}{\kappa^\sigma(\chi_\mu(\sigma))}\Gamma_\beta(\kappa^\sigma).
\]
Thus, $\ord_\pi\left(\Gamma_\beta(\kappa^\sigma)\right)$ is independent of $\sigma\in\Gal(J/k)$ because $\Gamma_\beta(\kappa)\in\pi\fO$ by assumption and $\kappa$ takes values in the group of roots of unity.
In particular, $\Gamma_\beta(\kappa^\sigma)\in \pi \fO$ for all $\sigma\in\Gal(J/k)$ under our hypothesis that $\Gamma_\beta(\kappa)\in\pi\fO$.

Let $N=\max(m,n)$.
Write $k_{N-1}$ to denote the $(N-1)$-th layer of the $\Zq$-extension $J/k$, and set $H=\Gal(k_{N-1}/k)$.
Let $y\in (1+q\Zq)^2$.
We have
\begin{align*}
 \sum_{\sigma\in H}\kappa^\sigma(y)^{-1}\Gamma_\beta(\kappa^\sigma)&= \sum_{\sigma\in H}\kappa^\sigma(y)^{-1}\int_{(1+q\Zq)^2}\kappa^\sigma(x) d\beta(x) \\
 &=\sum_{\sigma\in H}\int_{(1+q\Zq)^2}\kappa^\sigma(x/y)d\beta(x)\\
 &=\int_{(1+q\Zq)^2}\Tr_{k_{N-1}/k}\circ\kappa(x/y)d\beta(x) \\
 &=q^{N-1}\int_{y_1(1+q^m\Zq)\times y_2(1+q^n\Zq)}\kappa(x/y)d\beta(x) .
\end{align*}
Note that $q^{N-1}$ is a unit in $\fO$ (since $q\ne p$).
Therefore, $\Gamma_\beta(\kappa)\in\pi\fO$ implies that
\[
\int_{y_1(1+q^m\Zq)\times y_2(1+q^n\Zq)}\kappa(x/y)d\beta(x) \in\pi\fO.
\]
Let $x=(x_1,x_2)=y(1+q^mz_1,1+q^nz_2)=(y_1(1+q^mz_1),y_2(1+q^nz_2))$, where $z_1,z_2\in\Zq$.
Then 
\[
\kappa(x/y)=\kappa(1+q^mz_1,1+q^nz_2)=\kappa((1+q^m)^{z_1},(1+q^n)^{z_2})=\zeta_1^{x_1/y_1-1}\zeta_2^{x_2/y_2-1}.
\]
Thus, we deduce that
\[
\int_{y_1(1+q^m\Zq)\times y_2(1+q^n\Zq)}\zeta_1^{x_1/y_1}\zeta_2^{x_2/y_2}d\beta(x) \in\pi\fO.
\]
If we replace $y$ by $yt=(y_1t_1,y_2t_2)$ and $(\zeta_1,\zeta_2)$ by $(\zeta_1^{t_1},\zeta_2^{t_2})$ for any $t=(t_1,t_2)\in (1+q^M\Zq)^2$, the same containment holds.
Hence, summing over $t\in(1+q^M\Zq)^2/(1+q^m\Zq)\times (1+q^n\Zq) $, we deduce that
\[
\hat\beta_y(\underline\zeta^{y^{-1}})=\int_{y(1+q\Zq)^2}\zeta_1^{x_1/y_1}\zeta_2^{x_2/y_2}d\beta(x) \in\pi\fO.
\]

The converse follows from Lemma~\ref{lem:Gamma-Fourier} and the fact that the Gauss sum $\tau(\kappa)$ is a $\pi$-adic unit (which is a consequence of the fact that its conductor is coprime to $p$).
\end{proof}

\begin{proof}[Proof of Theorem~\ref{thm:valuation-gamma-transform}]
Without loss of generality, we assume that $n=\ord_\pi(R)=0$.
Let $\beta$ be as defined in the statement of Lemma~\ref{lem:beta}, and $w_K$ denote the number of elements in $\mu_K$ (which is coprime to $p>3$).
We have
\[
\frac{1}{w_K^2}\Gamma_\alpha(\kappa)=\Gamma_\beta(\kappa).
\]
Let us write 
\[
\alpha_{\eta y}=\alpha|_{\eta_1y_1(1+q^M\Zq)\times \eta_2y_2(1+q^M\Zq)}
\]
for $\eta=(\eta_1,\eta_2)\in \mu_{q-1}^2$ and $y=(y_1,y_2)\in (1+q\Zq)^2$.
Note that $\alpha_{\eta y}$ is an elliptic function measure since it is a restriction of $\alpha$.
Furthermore, we write $R_{\eta y}$ for the rational function on $E$ attached to $\alpha_{\eta y}$ (meaning that $\hat\alpha_{\eta y}= R_{\eta y}\circ\delta$ as functions on $\mu_{q^\infty}^2$).
As can be seen in the proof of Lemma~\ref{lem:EFM}, $R_{\eta y}$ takes values in $\fO$.
Let $\tilde R_{\eta y}$ denote the function $R_{\eta y}$ modulo $\pi$.

Suppose that the set of characters $\kappa$ with $\ord_\pi(\Gamma_\alpha(\kappa))= 0$ is not Zariski dense.
Note that for all $\kappa$, we have $\ord_\pi(\Gamma_\alpha(\kappa))=\ord_\pi(\Gamma_\beta(\kappa))$ by Lemma~\ref{lem:integral-elliptic-function} and the fact that $p\nmid w_K$.
Equivalently, the set of characters $\kappa$ such that $\Gamma_\beta(\kappa)\not\in \pi\fO$ is not Zariski dense.
By Lemma~\ref{lem:beta}, the set of elements $Q\in E_{q^\infty}$ such that
\[
\sum_{\eta\in(\mu_{q-1}/\mu_K)^2} \tilde R_{\eta y}([\eta^{-1}]\circ Q)\ne 0
\]
is not Zariski dense.

Applying Theorem~\ref{thm:Hida-Lam}, it follows that $\tilde R_{\eta y}$ is a constant function.
Let $c_{\eta y}$ denote a constant of $\fO$ lifting $\tilde R_{\eta y}$ and let $\delta_0$ denote the Dirac measure of $\Zq^2$ concentrated at $(0,0)$.
By definition, the Fourier transform $\hat\delta_0$ sends all $\underline{\zeta}\in \mu_{q^\infty}$ to $1$.
Therefore, the Fourier transform of $\alpha_{\eta y}-c_{\eta y}\delta_0$ takes values in $\pi\fO$.
In particular,
\[
\ord_{\pi}(\alpha_{\eta y}-c_{\eta y}\delta_0)>0.\]
However, if we restrict the measure $\alpha_{\eta y}-c_{\eta y}$ to $(\Zq^\times)^2$, it agrees with $\alpha_{\eta y}$.
Thus,
\[
\ord_{\pi}(\alpha_{\eta y})=\ord_{\pi}(\alpha_{\eta y}-c_{\eta y}\delta_0)>0.\]
This contradicts our hypothesis that $\ord_\pi(R)=0$.
\end{proof}

\section{Proof of Theorem~\ref{thmA}}
\label{S:L}
In this section we apply Theorem~\ref{thm:valuation-gamma-transform} to study $\pi$-adic valuations of special values of $L$-functions and prove Theorem~\ref{thmA} stated in the introduction.

\subsection{Notation on ray class fields and CM elliptic curves}

We keep the notation introduced in Section~\ref{section: preliminaries}.
Recall that $K$ is a fixed imaginary quadratic field, and $H$ is its Hilbert class field.

\begin{Defi}
Let $\fa$ be an integral ideal of $K$.
\begin{itemize}
 \item We write $\sR(\fa)$ for the ray class field of $K$ with conductor $\fa$.
 \item Given another ideal $\fb$ of $K$ which is coprime to $\fa$, we write $(\fb,\sR(\fa))\in\Gal(\sR(\fa)/K)$ for the Artin symbol of $\fb$.
 \item Given a character $\rho$ on $\Gal(\sR(\fa)/K)$, we shall write $\rho(\fb)$ and $\rho\left((\fb,\sR(\fa))\right)$ interchangeably.
\end{itemize}
\end{Defi}

Recall from \S\ref{section: preliminaries} that $E$ is an elliptic curve with complex multiplication by $\cO_K$ with good reduction at the primes above $p$ and $q$.
Let $\omega_E$ denote the N\'eron differential for $E_{/{\sR(\ff)}}$ and $\mathcal{L}=\Omega_\infty \cO_K$ be its period lattice.
Note that $\Omega_{\infty}$ is uniquely determined up to a root of unity in $K$.

Given an ideal $\fb$ of $K$ coprime to $\ff$, there exists {$\Lambda(\fb)\in \sR(\ff)^\times$ such that
\begin{equation}
 \cL_{\fb} = \Lambda(\fb)\fb^{-1}\cL
\label{eq:defnLambda}
\end{equation}
is the lattice associated with $E^{(\fb, \sR(\ff))}$, as given by \cite[(16) on p.~42]{dS87} (see also \cite[D{\'e}finition, p.~198]{GS81}).
For simplicity, we shall write $E^{(\fb)}$ for the CM elliptic curve $E^{(\fb,\sR(\ff))}$} and denote by 
\[
\lambda(\fb):E\rightarrow E^{(\fb)}
\]
the unique isogeny given by \cite[(15) on p.~42]{dS87}.

Consider the complex analytic isomorphism of complex Lie groups
\begin{equation}
\label{eqn: xi iso}
\xi_\fb : \CC/\cL_\fb \xrightarrow{\sim}E^{(\fb)}(\CC) \ \textrm{given by} \ \xi_\fb(z) = \left( \wp(z,\cL_\fb), \wp^\prime(z,\cL_\fb)\right),
\end{equation}
where $\wp$ is the Weierstrass $\wp$-function and $\wp^\prime$ is the corresponding derivative.
We have the Weierstrass equation
\begin{equation}\label{eq:weierstrass}
 y^2=4x^3-g_2(\cL_\fb)x-g_3(\cL_\fb)
\end{equation}
describing $E^{(\fb)}$.

When $\fb=\cO_K$, we shall write $\xi_1$ in place of $\xi_{\cO_K}$.
We recall the following relation:
\begin{equation}\label{eq:xi-transform}
 \xi_\fb\left(\Lambda(\fb)z\right)=\lambda(\fb)(\xi_1(z))
\end{equation}
as discussed in \cite[commutative diagram (21) on p.~43]{dS87} and \cite[Proposition~4.10]{GS81}.

\subsection{Review on \texorpdfstring{$L$}{}-functions}
\begin{Defi}
\label{imprimitive L function}
Let $\fh$ be any integral ideal of $K$.
Let $\epsilon$ be any Hecke character of $K$ with conductor dividing some power of $\fh$.
The \emph{imprimitive $L$-function} of $\epsilon$ modulo $\fh$ is defined as follows
\[
L_{\fh}(\epsilon,s) = \sum_{\gcd(\fa, \fh)=1} \frac{\epsilon(\fa)}{(N\fa)^s}.
\]
\end{Defi}
Let $\epsilon$ be a Hecke character over $K$ {of infinity type $(a,b)$}.
Denote by $L(\epsilon,s)$ the \emph{primitive Hecke $L$-function} of $\epsilon$.
Recall that the imprimitive (or partial) $L$-function differs from the primitive (or classical) $L$-function by a finite number of Euler factors.
We can further define the \emph{primitive algebraic Hecke $L$-function},
\[
L^{\alg}(\overline\epsilon) := \frac{L\left({\overline\epsilon,a+b}\right)}{{(2\pi)^b\Omega_{\infty}^{b-a}} }.
\]
If $\Psi$ and $\kappa$ are as in the statement of Theorem~\ref{thmA}, then
\[
L^{\alg}\left(\overline{\Psi\kappa}\right)=L^{\alg}\left(\overline{\varphi^{k-j}\kappa}\chi_0 N^j\right)= \frac{L\left(\overline{\varphi^{k-j}\kappa}\chi_0 N^j,k+j\right)}{(2\pi)^j\Omega_{\infty}^{k-j} }=\frac{L\left(\overline{\varphi^{k-j}\kappa}\chi_0,k\right)}{(2\pi)^j\Omega_{\infty}^{k-j} },
\]
where $\varphi$ and $\chi_0$ are given as in \S\ref{section: preliminaries}.

Henceforth, we assume that $\kappa$ is of conductor $\fq^{m+1} {\fq^*}^{n+1}$ and set $ F_{m,n}=\sR(\fh)$ with $\fh = \fg \fq^{m+1} {\fq^*}^{n+1}$.
Let $\fg$ be an auxiliary principal ideal that is divisible $\ff$ and is relatively prime to $pq$.
Then $\upsilon=\kappa\overline\chi_0$ is a character of $\Gal\left( \sR(\fh)/K\right)$.
The imprimitive $L$-function of $\overline{\upsilon \varphi^{k-j}}$ modulo $\fh$ can be written as
\[
L_{\fh}\left(\overline{\upsilon \varphi^{k-j}},s \right) = \sum_{\tau \in \Gal(\sR(\fh)/K)} \overline{\upsilon}(\tau)\sum_{(\fb,\sR(\fh))=\tau} \frac{\overline{\varphi^{k-j}}(\fb)}{(N\fb)^{s}},
\]
where the second sum runs over integral ideals $\fb$ of $\cO_K$ such that $\gcd(\fb,\fh)=1$.
We define the following partial imprimitive L-functions:
\begin{Defi}
Let $\fh$ and $\varphi$ be as above.
For $\tau\in\Gal(\sR(\fh)/K)$, we define
\[
L_{\fh}\left(\overline{\varphi^{k-j}}, s , (\fb,\sR(\fh))\right)=\sum_{\substack{\fb\unlhd\cO_K\\(\fb,\sR(\fh))=\tau\\\gcd(\fb,\fh)=1}} \frac{\overline{\varphi^{k-j}}(\fb)}{(N\fb)^{s}} .
\]
\end{Defi}
In particular, we have
\[
L_{\fh}\left(\overline{\upsilon\varphi^{k-j}},s \right) = \sum_{\tau \in \Gal(\sR(\fh)/K)} \overline{\upsilon}(\tau)L_\fh(\overline{\varphi^{k-j}}, s , \tau).
\]

\begin{rem}
The (primitive and imprimitive) $L$-functions we have discussed so far only converge on some right half-plane.
However, they admit analytic continuations to the entire complex plane.
In order to prove Theorem~\ref{thmA}, we shall relate $L^{\alg}\left(\overline{\upsilon\varphi^{k-j}}\right)$ to Gamma transforms of certain elliptic function measure that we construct in the following subsection.
\end{rem}

Let $F = \sR(\fg q)$ and write $\Delta =\Gal(F/K)$.
Since $\ff\mid \fg$, we have
\[
F = \sR(\fg q) = K\left(j(E), h(E_{\fg q})\right) = H\left(x(E_{\fg q})\right).
\]
Here $h$ denotes a Weber function and we may choose this to be the $x$-coordinates on a Weierstrass model for the elliptic curve.
Set $F_\infty=\bigcup_{n\ge1} \sR(\fg q^n)$; this is a $\ZZ_q^2$-extension of $F$.
Recall that $K_\infty$ is the $\Zq^2$-extension of $K$, we fix an isomorphism 
\[
\Gal(F_\infty/K)\simeq \Gal(F/K) \times \Gal(F_\infty/F) \simeq \Gal(F/K) \times \Gal(K_\infty/K) \simeq \Delta \times \Zq^2.
\]
By definition, $\upsilon=\kappa\overline\chi_0$ is a character of $\Gal(\sR(\ff)\cdot K_\infty/K)$, which is a quotient of $\Gal(F_\infty/K)$.
Our hypothesis that $q\nmid [\sR(\ff):K]$ allows us to regard $\kappa$ (resp. $\overline\chi_0$) as a character of $\Zq^2$ (resp. $\Delta$).
Then $\kappa$ (resp. $\upsilon$) may be regarded as a character of $\Gal(F_{m,n}/F)$ (resp. $\Delta\times \Gal(F_{m,n}/F)$).

\begin{Defi}
Given an ideal $\fc$ of $\cO_K$ that is coprime to $\fh$, let $\tau_\fc$ denote $(\fc,F_{m,n})=(\fc,\sR(\fh))$.
\end{Defi}

We conclude this subsection with the following lemma on the Galois action on partial imprimitive $L$-values.

\begin{Lemma}
\label{lem:action-on-L}
Let $\fb$ be an ideal of $\cO_K$ coprime to $\fh$ such that $(\fb,\sR(\ff))=1$.
For any $\rho\in \fh^{-1}\cL/\cL$ and any integral ideal $\fc$ of $\cO_K$ that is coprime to $\fh$, we have
\[
\tau_\fb\cdot \frac{L_{\fh}\left(\overline{\varphi^{k-j}}, k ,\tau_\fc\right)}{(2\pi)^j\rho^{k-j}}=\frac{L_{\fh}\left(\overline{\varphi^{k-j}}, k ,\tau_{\fb\fc} \right)}{(2\pi)^j\rho^{k-j}}.
\]
\end{Lemma}

\begin{proof}
Equation \eqref{eq:L-Eisen} in the appendix tells us that
\begin{equation}\label{eq:L-Eisen0}
\frac{L_{\fh}\left(\overline{\varphi^{k-j}}, k, \tau_\fc\right)}{(2\pi)^j\rho^{k-j}} 
= \frac{(N\fh\sqrt{d_K})^{-j}\Lambda(\fc)^{k-j} }{(k-1)!\varphi(\fc)^{k-j}} E_{j,k}\left(\rho, \cL \right)^{\tau_\fc}.
\end{equation}
Since $\tau_\fb$ acts trivially on $\sR(\ff)$, we deduce that
\begin{align*}
\tau_\fb\cdot \frac{L_{\fh}\left(\overline{\varphi^{k-j}}, k, \tau_\fc\right) }{(2\pi)^j\rho^{k-j}}&=\frac{ {(N\fh\sqrt{d_K})^{-j}}\Lambda(\fc)^{k-j}}{(k-1)!\varphi(\fc)^{k-j}} E_{j,k}\left( \rho, \cL \right)^{\tau_{\fb\fc}}.
\end{align*}
On replacing $\fc$ by $\fb\fc$ in \eqref{eq:L-Eisen0}, we have
\[ \frac{L_{\fh}\left(\overline{\varphi^{k-j}}, k, {\tau_{\fb\fc}}\right) }{(2\pi)^j\rho^{k-j}}
=\frac{(N\fh\sqrt{d_K})^{-j}\Lambda(\fb\fc)^{k-j}}{(k-1)!\varphi(\fb\fc)^{k-j}} E_{j,k}\left( \rho, \cL \right).\]

The hypothesis that $(\fb,\sR(\ff))=1$ implies that $\varphi(\fb)=\Lambda(\fb)$ by \cite[(18) on p.~42]{dS87}.
Thus, equation (17) in \emph{op. cit.} tells us that
\[
\frac{ \Lambda(\fb\fc)}{\varphi(\fb\fc)}=\frac{ \Lambda(\fc)^{\tau_{\fb}}\Lambda(\fb)}{\varphi(\fc)\varphi(\fb)}=\frac{ \Lambda(\fc)}{\varphi(\fc)}.
\]
Hence the result follows.
\end{proof}

\subsection{A rational function with a canonical divisor}
\label{S:elliptic-function}
The goal of this section is to generalize the construction of a rational function on a CM elliptic curve from \cite[Section~6.3]{Lam15}.
In order to consider Hecke characters of more general infinity-type, we introduce a new derivative operator, which did not make an appearance in \emph{loc. cit.}
This allows us to carry out step (3) outlined in the introduction.
The notation introduced in the previous section will continue to be utilized.

Let $\fb$ be an integral ideal of $K$ that is coprime to $\ff$.
We fix an auxiliary ideal $\fa$ of $\cO_K$ that is coprime to $6\fh$ and that $(\fa,\sR(\ff))=1$.
Define the rational function $\zeta_{\fb,\fa}$ on $E^{(\fb)}$ by 
\begin{equation}
\label{eqn: zeta rational function}
\zeta_{\fb,\fa}(P) = \prod_Q
\left( x(P)-x(Q)\right)^{-1},
\end{equation}
where $Q$ runs over a set of representatives of $E_{\fa}^{(\fb)}\setminus \{0\} \pmod{\pm 1}$.
There exists a constant $c(\fb,\fa)\in H^\times$ such that the function \[\gamma_{\fb,\fa}(P) := c(\fb,\fa)\zeta_{\fb,\fa}(P)\]
has the property that for all $\beta\in \End\left(E^{(\fb)}\right)$ with $\gcd(\beta,\fa)=1$,
\[
\gamma_{\fb,\fa}(\beta(P)) = \prod_{R\in \ker(\beta)} \gamma_{\fb,\fa}(P \oplus R)
\](see \cite[Appendix]{coates91}).

We can write
\[
\cL_{\fb} = \ZZ \omega_{1,\fb} + \ZZ \omega_{2,\fb}
\]
such that $\frac{\omega_{1,\fb}}{\omega_{2,\fb}}$ lies in the upper half plane.
We define the constant (see \cite[(4) on p.~48]{dS87})
\[
A(\cL_{\fb}) := \frac{1}{2\pi i}\left(\omega_{1,\fb}\overline{\omega_{2,\fb}} - \overline{\omega_{1,\fb}}\omega_{2,\fb}\right).
\]
As in \cite[p.~57, (4)]{dS87}, let 
\[
\partial=-\frac{\partial}{\partial z},\quad \cD_\fb=-A(\cL_\fb)^{-1}\left(\overline{z}\frac{\partial}{\partial z}+\overline{\omega}_{1,\fb}\frac{\partial}{\partial \omega_{1,\fb}}+\overline{\omega}_{2,\fb}\frac{\partial}{\partial \omega_{2,\fb}}\right).
\]
For integers $0\le -j<k$, define the derivative operator $\Djk$ on $\CC(E^{(\fb)})$ by
\[
\Djk(f) = \cD_\fb^{-j}\partial^{k+j}\log f(z),
\]
where $z$ is a complex variable after identifying $E^{(\fb)}$ with $\CC/\cL_\fb$ via $\xi_\fb$ as given by \eqref{eq:xi-transform}.

\begin{Lemma}
\label{lemma: Lam14 3.1.4}
Let $Q$ be a primitive $\fh$-division point on $E$ and $\rho \in \fh^{-1}\cL\setminus \cL$.
Then there exist $\sigma\in\Gal(F_{m,n}/H)$ and $\zeta\in \mu_K$ (which we identify with $\Aut(E)$) such that
\[
Q = \zeta\left( \xi_1(\rho)^{\sigma}\right).
\]
Fix $\fc_0$ to be an ideal of $\cO_K$ coprime to $\fh$ such that $\tau_{\fc_0}=\sigma$.
Suppose that $(\fc_0,\sR(\ff))=1$.
Let $\fb$ and $\fc$ be ideals of $\cO_K$ coprime to $\fh$ with $(\fc,\sR(\ff))=1$.
Then
\begin{multline}
\Djk(\gamma_{\fb,\fa})\circ\lambda(\fb)(Q^{\tau_\fc}) = -(k-1)! \left( (N\fa)- \Lambda(\fa)^{k-j}\tau_\fa\right) \left(\frac{\varphi(\fb)}{ \Lambda(\fb)}\right)^{k-j}\times\\
\left(\frac{N\fh\sqrt{d_K}}{2\pi}\right)^j \frac{L_{\fh}\left(\overline{ \varphi^{k-j}}, k, \tau_{\fb\fc\fc_0}\right)}{(\zeta\rho)^{k-j}}.
\end{multline}
\end{Lemma}

\begin{proof}
By Class Field Theory, we have
\[
(\mathcal{O}_K/\fh)^\times/\mu_K \simeq \Gal(F_{m,n}/H)\hookrightarrow \Aut(E[\fh])\]
(see \cite[Chapter 2, proof of Theorem~2.3]{silverman}).
It follows that $\Aut(E[\fh])$ is generated by the image of $\Gal(F_{m,n}/H)$ and $\mu_K$.
The first assertion now follows just as in the proof of \cite[Lemma~3.1.4]{Lam-thesis} or \cite[Lemma~6.4]{Lam15}.


 In the appendix, we prove in \eqref{eq:app} that with $P=\xi_\fb(\Lambda(\fb)\rho)$
\[
\Djk(\gamma_{\fb,\fa})(P) = -(k-1)! \left( (N\fa)- \Lambda(\fa)^{k-j}\tau_\fa\right) \left(\frac{\varphi(\fb)}{ \Lambda(\fb)}\right)^{k-j}
\left(\frac{N\fh\sqrt{d_K}}{2\pi}\right)^j \frac{L_{\fh}\left(\overline{ \varphi^{k-j}}, k, \tau_{\fb}\right)}{\rho^{k-j}}.
\]
On replacing $P$ (resp. $\rho$) by $\zeta P$ (resp. $\zeta\rho$), we deduce that
\[
\Djk(\gamma_{\fb,\fa})(\zeta P) = -(k-1)! \left( (N\fa)- \Lambda(\fa)^{k-j}\tau_\fa\right) \left(\frac{\varphi(\fc)}{ \Lambda(\fc)}\right)^{k-j}
\left(\frac{N\fh\sqrt{d_K}}{2\pi}\right)^j \frac{L_{\fh}\left(\overline{ \varphi^{k-j}}, k, \tau_{\fb}\right)}{(\zeta\rho)^{k-j}}.\]
If we let $\tau_{\fc\fc_0}$ act on both sides of this equation, Lemma~\ref{lem:action-on-L} tells us that
\[
\Djk(\gamma_{\fb,\fa})(\zeta P^{\tau_{\fc\fc_0}} )= -(k-1)! \left( (N\fa)- \Lambda(\fa)^{k-j}\tau_\fa\right) \left(\frac{\varphi(\fc)}{ \Lambda(\fc)}\right)^{k-j}
\left(\frac{N\fh\sqrt{d_K}}{2\pi}\right)^j \frac{L_{\fh}\left(\overline{ \varphi^{k-j}}, k, \tau_{\fb\fc\fc_0}\right)}{(\zeta\rho)^{k-j}}.\]
The result now follows from \eqref{eq:xi-transform}.
\end{proof}

Define 
\[
\rho_{m,n}=\frac{\Omega_{\infty}}{g\nu^{m+1}{\nu^*}^{n+1}}\in \CC^\times,
\]
where $g$, $\nu$, $\nu^*$ are fixed generators of $\fg$, $\fq$ and $\fq^*$ respectively (such generators exist since these ideals are assumed to be principal).
Then $\xi_1(\rho_{m,n})$ is a primitive $\fh$-division point of $E$ (since $\fh=\fg\fq^{m+1}{\fq^*}^{n+1}$).

Let $V$ (respectively $Q_{m,n}$) be a fixed primitive $\fg$-division (respectively $\fq^{m+1} {\fq^*}^{n+1}$-division) point on $E$.
By Lemma~\ref{lemma: Lam14 3.1.4}, there exist $\zeta\in \mu_K$ and $\sigma_0=\tau_{\fc_0}$, where $\fc_0$ is an ideal of $K$, coprime to $\fh$, depending on $V$ and $Q_{m,n}$, such that
\[
V\oplus Q_{m,n}=\zeta(\xi_1(\rho_{m,n})^{\sigma_0}).
\]
Since $(\fg,q)=1$, there is an  isomorphism of groups
\[
\Aut\left(E[\fg\fq^{m+1}(\fq^*)^{n+1}]\right)\simeq
\Aut\left(E[\fg]\right)\times
\Aut\left(E[\fq^{m+1}(\fq^*)^{n+1}]\right),
\]
which in turn induces the decomposition
\[
\Gal(F_{m,n}/H)\simeq\Gal(F_{m,n}/\sR(\fg))\times\Gal(\sR(\fg)/H).
\]
Therefore, we may choose $V$ so that $(\fc_0,\sR(\ff))=1$ for all $m$ and $n$.

By Lemma~\ref{lemma: Lam14 3.1.4}, given any ideals $\fb$ and $\fc$ of $\cO_K$ coprime to $\fh$ such that $(\fc,\sR(\ff))=1$, we have
\begin{align}&\ \Djk(\gamma_{\fb,\fa})\circ \lambda(\fb)\left((V\oplus Q_{m,n})^{\tau_\fc}\right)\notag\\
 =&\ -(k-1)!\left(N(\fa) - \Lambda(\fa)^{k-j} \tau_\fa\right)\left(\frac{\varphi(\fb)^{k-j}}{\Lambda(\fb)^{k-j}}\left(\frac{N\fh \sqrt{d_K}}{2\pi}\right)^j \cdot\frac{L_{\fh}\left(\overline{\varphi^{k-j}}, k ,\tau_{\fb\fc\fc_0}\right)}{(\zeta\rho_{m,n})^{k-j}}\right)\notag\\
 =&\ \frac{-(k-1)!\varphi(\fb)^{k-j}}{\Lambda(\fb)^{k-j}}\left(\frac{N\fh \sqrt{d_K}}{2\pi}\right)^j\cdot\frac{N(\fa)L_{\fh}\left(\overline{\varphi^{k-j}}, k ,\tau_{\fb\fc\fc_0} \right)-\varphi(\fa)^{k-j}L_{\fh}\left(\overline{\varphi^{k-j}}, k ,\tau_{\fa\fb\fc\fc_0}\right)}{(\zeta\rho_{m,n})^{k-j}}.
 \label{eq:temp-formula}
\end{align}

We fix $\{\fb_i:i\in I\}$ to be a set of representatives of integral ideals in $K$ such that $\Gal(\sR(\ff)/K)=\{(\fb_i,\sR(\ff)):i\in I\}$.
Recall that $\Gal(F_{m,n}/K)\simeq \Delta\times \Gal(F_{m,n}/F)$, where $\Delta=\Gal(F/K)$.
Then,
\[
[\sR(q\fg):\sR(\ff)]\sum_{\sigma\in\Gal(F_{m,n}/\sR(\fg))}\kappa^{-1}(\sigma)\sum_{i\in I}\chi_0(\fb_i)=\sum_{\eta\in\Gal(F_{m,n}/K)}\upsilon^{-1}(\eta).
\]
Let us regard  $\kappa$ as a character of $\Gal(F_{m,n}/\sR(\fg))\simeq \Gal(F_{m,n}/F)\times \Gal(F/\sR(\fg))$ sending the elements of $\Gal(F/\sR(\fg))$ to $1$.
We deduce from \eqref{eq:temp-formula} that
\begin{equation}
\begin{split}
\label{rational function and L function expression}
&\ \sum_{\sigma\in\Gal(F_{m,n}/\sR(\fg))}\kappa^{-1}(\sigma)\sum_{\delta\in\Gal(\sR(\fg)/\sR(\ff)),i\in I}\frac{\chi_0(\fb_i)\Lambda(\fb_i)^{k-j}}{\varphi(\fb_i)^{k-j}}\Djk(\gamma_{\fb_i,\fa})\circ\lambda(\fb_i)(V^\delta\oplus Q_{m,n}^\sigma) \\
=&\ -(k-1)!\left( \frac{N\fh \sqrt{d_K}}{2\pi}\right)^j \sum_{\eta\in\Gal(F_{m,n}/K)}\upsilon^{-1}(\eta)\frac{N(\fa)L_{\fh}\left(\overline{\varphi^{k-j}}, k ,\eta \tau_{\fc_0} \right)-\varphi(\fa)^{k-j} L_{\fh}\left(\overline{\varphi^{k-j}}, k ,\eta\tau_{\fa\fc_0}\right)}{(\zeta\rho_{m,n})^{k-j}}\\
=&\ -(k-1)!\left( \frac{N\fh \sqrt{d_K}}{2\pi}\right)^j\frac{N(\fa)\upsilon(\sigma_0)L_\fh\left(\overline{\varphi^{k-j}\upsilon},k\right)-\varphi(\fa)^{k-j}\upsilon(\sigma_0\tau_\fa)L_\fh\left(\overline{\varphi^{k-j}\upsilon},k\right)}{(\zeta\rho_{m,n})^{k-j}}\\
=&\ -(k-1)!\left( \frac{N\fh \sqrt{d_K}}{2\pi}\right)^j\left(N(\fa)-\varphi(\fa)^{k-j}\upsilon(\tau_\fa)\right)\upsilon(\sigma_0)\frac{L_\fh\left(\overline{\varphi^{k-j}\upsilon},k\right)}{(\zeta\rho_{m,n})^{k-j}}.
\end{split}
\end{equation}

The above calculations lead us to define the following rational function on $E$.

\begin{Defi}
\label{defi: rational function main}
Let $\fa$ be an ideal of $\cO_K$ chosen as above.
Let $V$ be a primitive $\fg$-division point of $E$, we define a rational function on $E$
 sending $P\in E$ to 
 \[
\rational(P)=\sum_{\delta\in\Gal(\sR(\fg)/\sR(\ff)),i\in I}\frac{\chi_0(\fb_i)\Lambda(\fb_i)^{k{-j}}}{\varphi(\fb_i)^{k-j}}\Djk(\gamma_{\fb_i,\fa})\circ\lambda(\fb_i)( V^\delta\oplus P).
\]
\end{Defi}

\subsection{Gamma transforms and \texorpdfstring{$L$}{}-values}
We can associate with $\rational$ an elliptic function measure, $\alpha$ on $\Zq^2$ via Lemma~\ref{lem:EFM}.
The measure $\alpha$ depends on $\Psi$ and our choice of $\fa$ and $V$.
We further define $\alpha^* = \alpha|_{(\Zq^\times)^2 }$.

We now relate the Gamma transform of $\alpha^*$ to special values of imprimitive algebraic $L$-functions.
Recall that $p$ is a rational prime satisfying $(p) = \fp {\fp^*}$ in $\cO_K$ with $\fp \neq \fp^*$ and $\gcd(p, 6q)=1$.
As before, set $\pi$ to be the uniformizer of the local field $k$, which is a finite unramified extension of $\QQ_p$ containing $\QQ_p(E_{q\fg})$.

\begin{Lemma}\label{thm: Lam15 6.6}
Let $\kappa$ be as before.
Then
\[
\ord_{\pi}\left( \Gamma_{\alpha^{*}} ( \kappa)\right) = \ord_{\pi}\left((k-1)!\left( N(\fa) - {\varphi}(\fa)^{k-j}\upsilon(\tau_\fa)\right)L_{\fh}^{\alg}(\overline{\upsilon\varphi^{k-j}}) \right),
\]
where $\upsilon=\kappa\overline\chi_0$.
\end{Lemma}

\begin{proof}
Let $\underline{\zeta}=(\zeta_1, \zeta_2)\in \mu_{q^\infty}^2$ and set $Q_{m,n} = \delta(\underline{\zeta})$ in our construction above.
Using Lemma~\ref{lem:Gamma-Fourier} in conjunction with \eqref{rational function and L function expression}, yields
\begin{align*}
\Gamma_{\alpha^*}(\kappa) &= \tau(\kappa)\sum_{\sigma\in \Gal(F_{m,n}/\sR(\fg))} \chi^{-1}(\sigma)\rational(Q_{m,n}^\sigma)\\
&={-(k-1)!\tau(\kappa)\left( \frac{N\fh \sqrt{d_K}}{2\pi}\right)^j}\left(N(\fa)-\varphi(\fa)^{k-j}\upsilon(\tau_\fa)\right)\upsilon(\sigma_0)\frac{L_\fh\left(\overline{\varphi^{k-j}\upsilon},k\right)}{(\zeta\rho_{m,n})^{k-j}}.
\end{align*}
Standard facts about Gauss sums tell us that $\ord_{\pi}(\tau(\kappa))= 0$ since the conductor of $\kappa$ is coprime to $p$.
Finally, as $\upsilon$ is a finite character, $\upsilon(\sigma_0)$ is a root of unity.
By our choice of $\fh$, we also know that $N\fh$ is coprime to $p$.
This completes the proof of the lemma.
\end{proof}
We now study the factor $N(\fa) - {\varphi}(\fa)^{k-j}\upsilon(\tau_{\fa})$.
Recall that $\tau_\fa$ denotes $(\fa,F_{m,n})$, and thus depends on $m$ and $n$, a priori.
However, we may regard it as an element of $\Gal(F_\infty/K)$ since the Artin symbols $\tau_\fa$ are compatible under restriction as $\fh$ varies over ideals dividing $\fg q^\infty$.

\begin{Lemma}
\label{lemma: factor has trivial pi-adic valuation}
For a Zariski dense set of $\kappa$, we have 
\[
\ord_\pi\left(N(\fa)-\varphi^{k-j}(\fa)\upsilon(\tau_\fa)\right)=0.
\]
\end{Lemma}

\begin{proof}
Suppose the contrary.
Let $(\zeta_1,\zeta_2)={(\kappa(\tau_{\fa,\fq}),\kappa(\tau_{\fa,{\fq^*}}))}\in \mu_{q^\infty}^2$, where $\tau_{\fa,\fl}$ denotes the restriction of $\tau_\fa$ to $\Gal(\sR(\fg\fl^\infty)/\sR(\fg))$.
Then, 
\[
\ord_\pi\left(N(\fa)-{\varphi}^{k-j}(\fa){\upsilon}(\tau_\fa)\right)=0
\]
if and only if 
\[
N(\fa){\varphi}(\fa)^{j-k}\not\equiv \zeta_1\zeta_2\mod \pi\fO
\]
since $(\fa,\sR(\ff))=1$, which implies that $\chi_0(\tau_\fa)=1$.
Note that the left-hand side is independent of $\kappa$.
In particular, this condition is invariant under the map $(\zeta_1,\zeta_2)\mapsto(\zeta_1,\zeta_2)^{p^r}$, where $p^r$ is the cardinality of the residue field of $\fO$.

Our assumption that the set of $\kappa$ satisfying the stated property above is not Zariski dense allows us to apply Lemma~\ref{lem:Hida}.
Let $P$ be the power of $p^r$ given by the said lemma.
In particular, under the isomorphism $\mu_{q^\infty}^2\simeq (\QQ_q/\Zq)^2$, there exists an arbitrary large $n$ such that 
\begin{equation}\label{eq:mod-pi}
 N(\fa){\varphi}(\fa)^{j-k}\equiv \zeta_1\zeta_2\mod \pi\fO
\end{equation}
for all $(\zeta_1,\zeta_2)$ which can be identified with $\displaystyle\left(\frac{P^x}{q^n},\frac{P^{y}}{q^n}\right)$, where $x,y\in \ZZ$.
In particular, Remark~\ref{rk:Hida} tells us that there are $q^{2(n-v)}$ such elements, where $v=\ord_q(P-1)$.

Note that $q$-power roots of unity modulo $\pi\fO$ are distinct since $p\ne q$.
Suppose that the left-hand side of \eqref{eq:mod-pi} modulo $\pi$ is a $q^m$-th root of unity, where $m<n$.
Then, for each $q^n$-th root of unity $\zeta_1$, there are exactly $q^{m}$ choices of $q^n$-th roots of unity $\zeta_2$ such that \eqref{eq:mod-pi} holds.
This gives us at most $q^{n+m}$ choices of $(\zeta_1,\zeta_2)\in\mu_{q^n}^2$.
But this is a contradiction as soon as $n+m<2(n-v)$.
\end{proof}

\begin{rem}
\label{the Zariski dense set is open}
For a given ideal $\fa$, denote the Zariski dense set of characters described in Lemma~\ref{lemma: factor has trivial pi-adic valuation} by $Z_{\fa}$.
This set is defined by the equation 
\[
f(\fa) := N(\fa)-{\varphi}^{k-j}(\fa){\upsilon}(\tau_\fa) \not\equiv 0 \mod \pi\fO.
\]
But note that $Z_{\fa} = \bigcup_{m\in \pi\fO} Z_{m,\fa}$ where each $Z_{m,\fa}$ is defined by equation
\[
f_{m}(\fa) := f(\fa)- m \neq 0,
\]
as $m$ varies over elements of $\pi\fO$.
Since each $Z_{m,\fa}$ is Zariski open, we have that $Z_{\fa}$ is Zariski open.
\end{rem}

Once we combine Lemmas~\ref{thm: Lam15 6.6} and \ref{lemma: factor has trivial pi-adic valuation} with Theorem~\ref{thm:valuation-gamma-transform}, Theorem~\ref{thmA} follows.

\section{Proof of Theorem~\ref{thmB}}
We continue employing the notation introduced in \S\ref{S:L}.
Throughout this section, we assume that $j=0$.
In addition, we assume that the character $\chi_0$ of $\Gal(\sR(\ff)/K)$ from \S\ref{section: preliminaries} satisfies
\begin{equation}
\label{eq:condition-rho}
\ord_\pi\left(\sum_{i\in I}\frac{\chi_0(\fb_i)}{\varphi(\fb_i)^k}\right)=0.
\end{equation}

\begin{rem}
\label{remark: sufficient conditions for technical hypothesis}
Note that the $\pi$-adic valuation in \eqref{eq:condition-rho} is always non-negative since $\fb_i$ are coprime to $\fp$.
Suppose that $p\nmid {[\sR(\ff):K]}$, then there exists at least {a character $\rho$ of $\Gal(\sR(\ff)/K)$ such that $\rho\chi_0$ satisfies \eqref{eq:condition-rho}.} 
Indeed, 
{\[
\sum_{\rho\in\widehat{\Gal(\sR(\ff)/K)}}\sum_{i\in I}\frac{\rho\chi_0(\fb_i)}{\varphi(\fb_i)^k}=[\sR(\ff):K].
\]
Therefore, if $\ord_\pi([\sR(\ff):K])=0$}, at least one of the summands should have zero $\pi$-adic valuation.
\end{rem}

The following lemma generalizes \cite[Lemma~6.7]{Lam15} and is crucial in our proof of Theorem~\ref{thmB}.

\begin{Lemma}
\label{lemma: Lam15 Lemma 6.7}
{Suppose that our auxiliary ideal $\fa$ is chosen so that $\gcd\left(\fa, 6q\fp \fg\prod_{i\in I}\fb_i\right)=1$, $(\fa,\sR(\ff))=1$}, and $\fa\equiv 1\mod \ff$.
Then,
\[
\ord_{\pi}\left( \rational\right) ={\ord_\pi\left((k-1)!\right)}.
\]
\end{Lemma}
\begin{proof}
Let us first recall the following facts proved in \cite[proof of Lemma~6.7]{Lam15}.
\begin{itemize}
 \item[(a)] The rational function $\Dnotk(\gamma_{\fb,\fa})$ on $E^{(\fb)}$ has poles of order $k$ at all the elements of $P\in E_\fa^{(\fb)}\setminus\{0\}$, with leading coefficient with respect to $z-z_P$ equal to $(k-1)!$.
 \item[(b)]Furthermore, $\Dnotk(\gamma_{\fb,\fa})$ has a pole of order $k$ at $P=0$, with leading coefficient with respect to $z$ equal to $N(\fa)-1$.
 \item[(c)]The poles described above are the only poles of $\Dnotk(\gamma_{\fb,\fa})$.
 \item[(d)]Let $x_\fb$ and $y_\fb$ be the functions sending a point $P\in E^{(\fb)}$ to its $x$- and $y$-coordinates given by the Weierstrass equation \eqref{eq:weierstrass}.
 The only zeros of the function $x_\fb(P)-x_\fb(R)$ are $P=R$ and $P=\ominus R$.
 If $x_\fb(R)\ne0$, these are simple zeros and the leading coefficient with respect to $z-z_P$ is given by $y_\fb(P)$.
\end{itemize}

Let $i\in I$.
Since $\fa$ is coprime to $\fb_i$, the isogeny $\lambda(\fb_i)$ induces an isomorphism $E_\fa\simeq E_\fa^{(\fb_i)}$.
Therefore, by (c), the poles of $\Dnotk(\gamma_{\fb_i,\fa})\circ \lambda(\fb_i)$ are precisely the elements in $E_\fa$.
Recall from Definition~\ref{defi: rational function main} that
\[
\rational(P) = \sum_{i\in I}\frac{\chi_0(\fb_i)\Lambda(\fb_i)^k}{\varphi(\fb_i)^k}{\sum_{\delta\in\Gal(\sR(\fg)/\sR(\ff))}}\Dnotk(\gamma_{\fb_i,\fa})\circ\lambda(\fb_i)( V^\delta\oplus P).
\]
In particular, the poles of $\rational(P)$ are given by $U\ominus V^\delta$, where $U\in E_\fa$ and $\delta\in \Gal(\sR(\fg)/{\sR(\ff)})$.

Let $P$ be a pole of $\Dnotk(\gamma_{\fb_i,\fa})\circ \lambda(\fb_i)$.
By \eqref{eq:xi-transform}, the leading coefficient of $\Dnotk(\gamma_{\fb_i,\fa})\circ \lambda(\fb_i)$ with respect to $z-z_P$ is that of $\Dnotk(\gamma_{1,\fa})$ multiplied by $\Lambda(\fb_i)^{-k}$, where $\gamma_{1,\fa}$ denotes the rational function on $E$ (so corresponding to the choice of $i$ gives $E^{(\fb_i)}=E$).
Consequently, by (a) the leading coefficient of $\rational$ with respect to $z-z_P$, when $P$ is the pole $U\ominus V^\sigma$ where $U\in E_\fa\setminus\{0\}$, is given by
\[
(k-1)!\sum_{i\in I}\frac{\chi_0(\fb_i)}{\varphi(\fb_i)^k},
\]
which has $\pi$-adic valuation equal to $\ord_\pi\left((k-1)!\right)$ by assumption \eqref{eq:condition-rho}.

Let $i\in I$, $\delta \in \Gal(\sR(\fg)/\sR(\ff))$ and $Q\in E_\fa \setminus\{0\}$.
By (d), the rational functions (on $E$) given by $x_{\fb_i}\circ \lambda(\fb_i)(P\oplus V^\delta)-x_{\fb_i}\circ \lambda(\fb_i)(Q)$ and $x(P\oplus V^\delta)-x(Q)$ (where $x$ denotes the $x$-coordinate function on $E$) have the same zeros.
Furthermore, by \eqref{eq:xi-transform}, the leading terms of these two rational functions differ by the constant $\Lambda(\fb_i)$.
Consequently, these two functions differ by a unit in $\fO$.
Therefore, as in \cite[proof of Lemma~6.7]{Lam15}, we can write
\[
\rational(P)=g(P)\prod_{\substack{\delta\in \Gal(\sR(\fg)/\sR(\ff)) \\ Q\in\left( E_\fa \setminus\{0\}\right)/ \pm1}}\left((x(P\oplus V^\delta)-x(Q)\right)^{-k},
\]
where $g$ is a rational function on $E$ belonging to
\[
\fO\Big[x\left(\lambda(\fb_i)(P\oplus V^\delta)\right), y\left(\lambda(\fb_i)(P\oplus V^\delta)\right) : i\in I, \delta\in \Gal(\sR(\fg)/\sR(\ff))\Big].
\]
In particular $\ord_\pi(g)\ge0$.

As has been established in \cite[proof of Lemma~6.7]{Lam15}, the functions $x(P\oplus V^\delta)-x(Q)$ take values in $\fO^\times$ for almost all $P$.
Furthermore, by comparing leading terms at $P=U\ominus V^\delta$, we deduce that $g$ takes values in $\fO^\times$ at these points.
Thus, $\ord_\pi(g)=0$, which concludes the proof.
\end{proof}

We can now prove Theorem~\ref{thmB}.
Let $\upsilon=\kappa\overline\chi_0$ as before.
By an argument similar to Lemma~\ref{lemma: factor has trivial pi-adic valuation} it suffices to prove the theorem for imprimitive values $L_{\fh}^{\alg}\left(\overline{\Psi\kappa}\right)$ because for almost all finite-order characters $\kappa$ of $\Gal(K_\infty/K)$, we have
\[
\ord_{\pi}\left( L^{\alg} \left({\overline{\Psi\kappa}}\right)\right) = \ord_{\pi}\left( L_{\fh}^{\alg}\left({\overline{\Psi\kappa}}\right)\right).
\]
Indeed, for any prime ideal $\mathfrak{r}$ of $K$ and for almost all characters $\kappa$,
\[
\ord_{\pi}\left(1 - \frac{{\overline{\Psi\kappa(\mathfrak{r})}}}{N(\mathfrak{r})^k} \right) =0
\]
as the $q$-power roots of unity modulo $\pi\fO$ are distinct since $p\ne q$.

Lemma~\ref{lemma: Lam15 Lemma 6.7} asserts that {$\ord_{\pi}\rational=\ord_\pi\left((k-1)!\right)$}.
In particular, the associated elliptic function measure $\alpha^*$ satisfies $\ord_\pi\alpha^*={\ord_\pi\left((k-1)!\right)}$.
Therefore, on combining Lemma~\ref{thm: Lam15 6.6} with Theorem~\ref{thm:valuation-gamma-transform}, we deduce that for a Zariski dense set of $\kappa$, we have
\[
\ord_{\pi}\left( \left( N(\fa) - {\varphi}(\fa)^k{\upsilon(\tau_\fa)}\right) L_\fh^{\alg}\left({\overline{\Psi\kappa}}\right)\right)=0.
\] 
The same argument as in Remark \ref{the Zariski dense set is open} shows that this Zariski dense set is also open.
Since the intersection of two open dense sets is open dense, there exists a dense set of characters $\kappa$ with
\[\ord_{\pi}\left( L^{\alg}\left(\overline{\Psi\kappa}\right)\right)=0.\]

\appendix
\section{appendix}
In this appendix we carry out a technical calculation required in the proof of Lemma~\ref{lemma: Lam14 3.1.4}.
For this calculation, we rely heavily on the work of de Shalit in \cite{dS87}.
In particular, we express special $L$-values in terms of logarithmic derivatives of rational functions.
We do so by relating both of these quantities to values of Eisenstein series.

\subsection{Relating rational functions to Eisenstein series}
As in the main text, let $K$ be an imaginary quadratic field and $H/K$ be the Hilbert class field of $K$.
Let $E_{/H}$ be a CM elliptic curve with CM by $\mathcal{O}_K$ and $\cL$ be the associated lattice.
Let $\fa$ and $\fb$ be ideals of $K$ such that $\fb$ is coprime to $6\ff$.
With respect to $\cL_\fb$, we can define an \emph{elliptic function}, denoted by $\Theta(z;\cL_\fb,\fa)$, as in \cite[Chapter~II, Section~2.3, (10) on p.~49]{dS87}.
Let $\xi_\fb$ be the isomorphism of complex Lie groups defined in \eqref{eqn: xi iso}.
It follows from \cite[(16) on p.~54]{dS87} that for any $z\in\CC$ with $P=\xi_\fb(z)\in E^{(\fb)}$,
\begin{equation}
 \Theta(z;\cL_\fb,\fa) =C_{\fb,\fa}\cdot \zeta_{\fb,\fa}(P)^{12},
\label{eq:Theta-zeta}
\end{equation}
where $\zeta_{\fb,\fa}(P)$ is the rational function introduced in \eqref{eqn: zeta rational function} and $C_{\fb,\fa}$ is some constant that is independent of $P$ and $z$ (the power of 12 appears because the product in \eqref{eqn: zeta rational function} is taken over $\fa$-torsions modulo $\pm1$, whereas the product in \cite[(16) on p.~54]{dS87} is taken over all non-trivial $\fa$-torsions, without modulo $\pm1$).

For integers $k\geq 1$ and $0\leq -j < k$, let $E_{j,k}(z,\cL_\fb)$ be the $(j,k)$-th \emph{Eisenstein series} associated to the lattice $\cL_\fb$ given as in \cite[(5) on p.~57]{dS87}.
Notice that when $k+j\ge3$, we have explicitly
\[
E_{j,k}(z,\cL_\fb) = (k-1)!A(\cL_\fb)^j{\sum_{w\in \cL_\fb}}'\frac{(\overline{z} + \overline{w})^{k-j}}{\abs{z+w}^{2k}}=(k-1)!A(\cL_\fb)^j{\sum_{w\in \cL_\fb}}'\frac{(\overline{z} + \overline{w})^k (z + w)^j}{\abs{z+w}^{2(k+j)}}.
\]
Here, the sum runs over all $w\in \cL_\fb$ except possibly $w=-z$ if $z\in \cL_\fb$.
Further, for each integral ideal $\fa$, we can define (see \cite[(5) on p.~57]{dS87})
\[
E_{j,k}(z;\cL_\fb,\fa) = (N\fa) E_{j,k}(z,\cL_\fb) - E_{j,k}(z, \fa^{-1}\cL_\fb).
\]

From \eqref{eq:Theta-zeta}, we deduce that, for $k\geq 1$, 
\begin{equation}
\label{Dk and Ez}
\begin{split}
12\Djk(\gamma_{\fb,\fa})(P) &= \cD_\fb^{-j}\partial^{k+j} \log \Theta(z;\cL_\fb,\fa)\\
&= -12 E_{j,k}(z,\cL_\fb, \fa)\quad \textrm{ by \cite[Chapter~II, Section~3.1, (7) on p.~58]{dS87}.}
\end{split}
\end{equation}

\subsection{Relating Eisenstein series to rational \texorpdfstring{$L$}{}-values}
Recall that $\ff$ is an ideal of $\cO_K$ that is divisible by the conductor of the Hecke character $\varphi$.
Let $\fm$ be a principal ideal of $\cO_K$ such that $\ff\mid \fm$.
Let $\fc$ be another ideal which is coprime to $\fm$.
Then for any $\Omega\in \CC^\times$ \cite[Chapter~II, Proposition~3.5, p.~62]{dS87} asserts that
\begin{equation}
\label{3.5 ds}
(N\fm^{-j})E_{j,k}\left(\Omega, \fc^{-1}\fm\Omega\right) = (k-1)! \left( \frac{\sqrt{d_K}}{2\pi}\right)^{j}\Omega^{j-k}\varphi(\fc)^{k-j} L_{\fm}(\overline{\varphi^{k-j}}, k, (\fc, \sR(\fm))).
\end{equation}
Let $\alpha\in\cO_K$ be a generator of our chosen principal ideal $\fm$.
We choose $\Omega\in\CC^\times$ in \eqref{3.5 ds} to be the period $\Omega_\infty$ so that
\[
\cL=\Omega_\infty \mathcal{O}_K.
\]
Let $\rho$ be the primitive $\fm$-division point on $\CC/\cL$ given by $\rho = \frac{\Omega_\infty}{\alpha}$.
Then,
\begin{align*}
E_{j,k}\left( \Omega_\infty, \fc^{-1}\fm\Omega_\infty \right) &= E_{j,k}\left( \rho\alpha, \fc^{-1}\fm\Omega_\infty \right)\\ 
&= E_{j,k}\left( \rho\alpha, \fc^{-1}\fm\cL \right)\\
&= \alpha^{j-k}E_{j,k}\left( \rho, \fc^{-1}\cL \right) \quad \textrm{by \cite[Proposition~3.3(i), p.~58]{dS87}}\\ 
&= \alpha^{j-k}\Lambda(\fc)^{k-j} E_{j,k}\left( \rho, \cL \right)^{(\fc, \sR(\fm))} \quad \textrm{by \cite[Proposition~3.3(iii), p.~58]{dS87}}, 
\end{align*}
where $\Lambda(\fc)\in H^\times$ is defined as in \eqref{eq:defnLambda}.

Combined with \eqref{3.5 ds}, the above calculation shows that
\begin{align}
\notag(k-1)! L_{\fm}\left(\overline{\varphi^{k-j}}, k, (\fc,\sR(\fm))\right) 
&= \left(\frac{\Lambda(\fc) \Omega_\infty}{\alpha \varphi(\fc)} \right)^{k-j}\left(\frac{2\pi}{N\fm\sqrt{d_K}}\right)^j E_{j,k}\left( \rho, \cL \right)^{(\fc,\sR(\fm))}\\
&= \left(\frac{\Lambda(\fc) \rho}{\varphi(\fc)}\right)^{k-j} \left(\frac{2\pi}{N\fm\sqrt{d_K}}\right)^j E_{j,k}\left( \Lambda(\fc)\rho, \cL_{\fc} \right) \quad \textrm{by \cite[(8), p.~58]{dS87}}
\label{eq:L-Eisen}.
\end{align}

\begin{rem}
In the special case when $H=K$ and $E$ is defined over $K$, (i.e., $K$ has class number 1) we know from \cite[(18) on p.~42]{dS87} that $\Lambda(\fc) = \varphi(\fc)$.
Moreover, it is also clear in this case that $\psi = \varphi$.
Therefore, on taking $j=0$, we obtain
\[
L_{\fm}\left(\overline{\psi^k}, k, (\fc,\sR(\fm))\right) = \frac{\rho^k}{(k-1)!} E_k\left( \psi(\fc)\rho, \cL_\fc \right)
\]
(c.f. \cite[Theorem~6.2]{Lam15}).
\end{rem}

\subsection{Relating rational functions to \texorpdfstring{$L$}{}-values}
\label{appendix final calculation}
Our final step is to combine the calculations in the previous two sections to relate the image of the operator $\Djk$ applied to our chosen rational function to the $(j,k)$-th Eisenstein series.
Let $P$ be an $\fm$-torsion on $E$.
We know from \eqref{Dk and Ez} that
\begin{align*}
\Djk(\gamma_{\fc,\fa})(P) &= - E_{j,k}(z;\cL_\fc,\fa)\\
&= -\left( (N\fa)E_{j,k}(z,\cL_\fc) - E_{j,k}(z,\fa^{-1}\cL_\fc)\right)\\
&= -\left( (N\fa)E_{j,k}(z,\cL_\fc) - \Lambda(\fa)^{k-j}E_{j,k}(z,\cL_\fc)^{(\fa,\sR(\fm))}\right) \ \textrm{by \cite[Prop.~3.3(iii), p.~58]{dS87}}\\
&= -\left( (N\fa)- \Lambda(\fa)^{k-j}(\fa,\sR(\fm))\right)E_{j,k}(z,\cL_\fc).
\end{align*}
Now, choosing $P=\xi_\fc\left(\Lambda(\fc)\rho\right)$,
we deduce that
\begin{align}
 \Djk(\gamma_{\fc,\fa})(P) = -(k-1)! \left( (N\fa)- \Lambda(\fa)^{k-j}(\fa,\sR(\fm))\right)& \left(\frac{\varphi(\fc)}{\rho \Lambda(\fc)}\right)^{k-j}\times\notag\\
 &\left(\frac{N\fm\sqrt{d_K}}{2\pi}\right)^j L_{\fm}\left(\overline{ \varphi^{k-j}}, k, (\fc,\sR(\fm))\right),\label{eq:app}
\end{align}
which is the formula that is utilized in the proof of Lemma~\ref{lemma: Lam14 3.1.4}.

\bibliographystyle{amsalpha}
\bibliography{references}
\end{document}